\documentclass[a4paper]{amsart}

\usepackage[%
backend=bibtex,%
style=alphabetic,%
natbib=false,%
isbn=false, 
doi=false, 
url=false, 
eprint=true, 
giveninits=true, 
backref=false, 
maxbibnames=99,%
]{biblatex}%

\AtEveryBibitem{%
  \ifentrytype{unpublished}{}{\clearfield{note}}%
  }%
\renewbibmacro{in:}{} 

\usepackage{amssymb,amsmath}%
\usepackage{doi}%
\usepackage{enumitem}%
\usepackage{eucal}%
\usepackage{graphicx}%
\usepackage{mathscinet}%
\usepackage{mathtools}%
\usepackage{tikz-cd}%
\usepackage{xparse}%

\usepackage{hyperref}%
\hypersetup{%
  colorlinks,%
  linkcolor={red!80!black},%
  citecolor={blue!80!black},%
  urlcolor={blue!80!black}%
}%

\usepackage{macros}

\begin{document}

\title%
[Stable $\infty$-categories for representation theorists]%
{Stable $\infty$-categories\\ for representation theorists}%

\author[G.~Jasso]{Gustavo Jasso}%

\address{%
  Mathematisches Institut, %
  Universität zu Köln, %
  Weyertal 86-90, %
  50931 Köln, %
  Germany}%

\email{gjasso@math.uni-koeln.de}%
\urladdr{https://gustavo.jasso.info}%



\begin{abstract}
  This survey is intended as an invitation to the theory of stable $\infty$-categories, addressed
primarily to mathematicians working in the representation theory of algebras and
related subjects.


\end{abstract}

\maketitle

\tableofcontents

\section*{Introduction}

This survey is intended as an invitation to the theory of stable
$\infty$-categories, addressed primarily to mathematicians working in the
representation theory of algebras and related subjects. We assume familiarity
with triangulated categories, but no background in algebraic topology or
homotopy theory is required. Familiarity with differential graded categories or
$A_\infty$-categories can help guide intuition, but it is not necessary for
following the exposition.

Similar to differential graded categories, stable $\infty$-categories can be
viewed as enhancements or, more precisely, refinements of triangulated
categories. A key difference is that differential graded categories enhance
triangulated categories of algebraic origin, whereas stable $\infty$-categories
naturally accommodate the broader class arising in topology. For readers
primarily interested in algebraic triangulated categories, this distinction may
be secondary. Instead, we focus on one of the primary advantages of stable
$\infty$-categories: Namely, that they are embedded in a robust
higher-categorical framework. In practice, this often makes conceptual and
categorical arguments easier to implement than in the settings of triangulated
or differential graded categories. This comes at the cost of the technical
complexity of the foundations of $\infty$-category theory, which we treat
largely as a black box in this survey.

This survey compiles results on $\infty$-categories and stable
$\infty$-categories, most of them drawn almost verbatim from Lurie's
foundational works~\cite{Lur09,Lur17,Lur18SAG}. We have selected topics that
highlight central features of the theory and that may be of particular interest
to representation theorists. Our main omission is a discussion of linear
structures on stable $\infty$-categories, a topic we hope to address elsewhere;
extensive information can be found in~\cite[Appendix~D]{Lur18SAG}.

We provide relatively little in the way of motivation or explicit examples.
Instead, we assume that the reader is already convinced of the usefulness of
triangulated categories and has examples of interest in mind. While some of the
most striking applications of stable $\infty$-categories so far occur in stable
homotopy theory, derived algebraic geometry, tensor-triangular geometry, and
number theory, in representation theory of algebras the theory has also enabled
important advances in cluster categorification in relation to Fukaya categories
of surfaces by Christ and
collaborators~\cite{Chr21a,Chr22,Chr22a,CHQ23,CHQ24,Chr25a,Chr25b,Chr25} as well
as the construction of very general forms of reflection functors in joint work
with Dyckerhoff and Walde~\cite{DJW19b,DJW21}, with the latter being used to
study (spectral) derived Picard groups by S{\'a}nchez in~\cite{San25}. An
$\infty$-categorical approach to Khovanov--Qi's Hopfological algebra~\cite{Kho16,Qi14} is pursued in~\cite{GJN26}.

This survey is structured as follows. In \Cref{sec:primer} we introduce the most
basic aspects of the theory of $\infty$-categories without delving into its
technical foundations. \Cref{sec:stable-cats} focuses on stable
$\infty$-categories, which are the main objects of interest in this article.
After recalling their definition and first properties, we discuss the notion of
idempotent completeness and later explain, using results of Cisinski, how to
construct stable $\infty$-categories from Frobenius exact categories. We then
outline some of the important features of the theory of compactly-generated and
presentable stable $\infty$-categories. Returning to the small setting, we
discuss smooth and proper stable $\infty$-categories. We also discuss
$t$-structures and realisation functors; the latter are obtained from the
universal properties enjoyed by derived $\infty$-categories. We also explain a
relation between ($t$-)injective objects in the sense of Lurie and large
cosilting objects. In the final part of this section we study recollements of
stable $\infty$-categories and illustrate the robustness of the theory in this
context.

\section{Primer on $\infty$-category theory}
\label{sec:primer}

This section consist of a brief introduction to the most basic ideas in
$\infty$-category theory; here, by $\infty$-category we mean a quasi-category in
the sense of Boardman and Vogt~\cite{BV73} (a kind of simplicial set). The
theory of $\infty$-categories has been developed extensively by
Joyal~\cite{Joy02,Joy08}, Lurie~\cite{Lur09,Lur17} and several others. By now,
there are textbook accounts of the foundational aspects of
theory~\cite{Cis19,Lan21,kerodon}, as well as excellent surveys~\cite{ACam16}
and more leisurely introductions~\cite{Gro20} that we wholeheartedly recommend
to the reader. It is also worth mentioning ongoing efforts by various groups of
mathematicians to develop model-independent\footnote{For a comparison of various
  models for the theory of $(\infty,1)$-categories, we refer the reader
  to~\cite{Ber18} (see also~\cite{Toe05}).} axiomatic approaches to the theory
of $\infty$-categories, see for example~\cite{RV22,CCW}. In this section we take
a naive approach and do not give the precise definition of an $\infty$-category
at all. Instead, we hope to provide the reader with a modicum of intuition
concerning the main features and peculiarities of the theory that we think
should be sufficient for appreciating the main concept that is discussed in this
survey, namely that of a stable $\infty$-category.

\subsection{From sets to $\infty$-groupoids}

In ordinary category theory, sets and functions are the most fundamental
objects. After all, the very notion of a category is obtained by abstracting the
associativity and unitality constraints of function composition. In contrast,
the most fundamental objects in $\infty$-category theory are the so-called
\emph{$\infty$-groupoids}, the precise definition of which is somewhat
technical, but the notion itself can be approached in a rather intuitive way.

Let us denote by $\Grpd$ the ordinary category of small groupoids. Recall that
there is an adjunction
\[
  \begin{tikzcd}
    \pi_0\colon\Grpd\rar[shift left]&\Set,\lar[shift left,hook']
  \end{tikzcd}
\]
where the left-adjoint functor $X\mapsto\pi_0(X)$ associates to a groupoid $X$
its set of \emph{connected components} (=isomorphism classes). Here, we
interpret sets as groupoids with only identity morphisms. From this
perspective, a groupoid $X$ `enhances' its set of objects by providing
additional information in the form of non-identity isomorphisms between its
objects, which are called the \emph{points} of the groupoid in this context.
Each groupoid $X$ admits a canonical decomposition
\[
  X\cong\coprod_{[x]\in\pi_0(X)}X[x],
\]
into its connected components
\[
  X[x]\coloneq\set{y\in X}[y\cong x].
\]
Moreover, up to equivalence, a connected component $X[x]$ of $X$ is determined by its
\emph{fundamental group}
\[
  \pi_1(X,x)\coloneqq X(x,x),
\]
in the sense that the inclusion of the point $x\in X$ induces an equivalence
between its connected component $X[x]$ and the groupoid with a single object $x$
with automorphism group $\pi_1(X,x)$. Thus, we can think of a groupoid as a
collection of points with a preferred notion of equivalence between them, namely
that of isomorphism.

Roughly speaking, an $\infty$-groupoid is some sort of generalisation of an
ordinary groupoid equipped not only with a notion of equivalence between its
points, but also with a notion of `equivalence between equivalences', a notion
of `equivalence between equivalences between equivalences', etc. Taking this
heuristic to its limit, given points $x$ and $y$ in an $\infty$-groupoid $X$,
there is an $\infty$-groupoid $X(x,y)$ whose points are the equivalences
$x\stackrel{\sim}{\to}y$. To an $\infty$-groupoid $X$ we may thus associate its
connected components $\pi_0(X)$. Moreover, if we define the \emph{loop space}
\[
  \Omega(X,x)=\Omega^1(X,x)\coloneqq X(x,x),\qquad x\in X,
\]
and, inductively,\footnote{Here, we abuse notation and identify the canonical
  isomorphism $\id[x]\in\Omega(X,x)$ with the point $x\in X$, compare
  with~\cite[Section~3.8.1]{Cis19}.}
\[
  \Omega^{n+1}(X,x)\coloneqq \Omega(\Omega^n(X,x),x),\qquad n\geq1,
\]
we obtain the \emph{$n$-th homotopy group}
\[
  \pi_n(X,x)\coloneqq\pi_0(\Omega^n(X,x)),\qquad n\geq1.
\]
Again, each $\infty$-groupoid $X$ admits a canonical decomposition
\[
  X\simeq\coprod_{[x]\in\pi_0(X)}X[x]
\]
into its connected components, and all points in the same connected component
have (non-canonically) isomorphic homotopy groups. The use
of terminology from homotopy theory is justified by the fact that, up to the
appropriate notion of equivalence, every $\infty$-groupoid arises as the
so-called \emph{Poincar{\'e} $\infty$-groupoid}\footnote{Materialised as the
  singular simplicial set in the implementation of $\infty$-categories as
  quasicategories~\cite[Remark~7.8.11]{Cis19}.} of a CW complex that is,
moreover, unique up to homotopy equivalence; this is part of the content of
Grothendieck's famous Homotopy Hypothesis~\cite{Gro22}, which is either a
tautology or a theorem depending on the precise definition of an
$\infty$-groupoid that one chooses.\footnote{We do not recall the precise stament
  as it does not play a role in the sequel.} In particular, the homotopy groups
$\pi_n(X,x)$, $n\geq1$, are indeed groups, and these groups are abelian whenever
$n\geq2$~\cite[Proposition~3.8.2]{Cis19}. For this reason, it is also customary
to refer to $\infty$-groupoids as \emph{spaces}, and we shall conflate these two
terms in the sequel.\footnote{Another equivalent term, introduced and
  popularised by Clausen and Scholze, is \emph{anima}.} A standard reference for
the homotopy theory of $\infty$-groupoids, implemented as Kan complexes, is the
textbook by Goerss and Jardine~\cite{GJ99}, which approaches the concept from an
algebro-topological perspective. The higher-categorical properties of
$\infty$-groupoids are discussed in detail in~\cite{Cis19} and~\cite{kerodon}.

\subsection{Basic $\infty$-category theory}

The collection $\Grpd_\infty$ of all small $\infty$-groupoids, while no longer
an $\infty$-groupoid, is yet another kind of mathematical object called an
$\infty$-category. The crucial point here is that the maps $X\to Y$ between
$\infty$-groupoids $X$ and $Y$ are the points of an $\infty$-groupoid
$\Map{X}{Y}$.

More generally, an $\infty$-category $\C$ is a category-like structure in which,
for each pair of objects $x,y\in\C$, there is an $\infty$-groupoid or, rather,
space of maps $\Map[\C]{x}{y}$. The upshot is that, given a map $f\colon x\to y$
in $\C$, we may consider the homotopy groups
\[
  \pi_n(\Map[\C]{x}{y},f),\qquad n\geq1,
\]
which encode information about the automorphisms of the map $f\colon x\to y$
itself, as well as the set of connected components
\[
  \Hom[\C]{x}{y}=\Hom[\Ho(\C)]{x}{y}\coloneqq\pi_0(\Map[\C]{x}{y}).
\]
The latter notation is justified by the fact that there is an ordinary category
$\Ho[\C]$, called the \emph{homotopy category} of $\C$, with the same objects as
$\C$ and with the above sets of maps. The fact that $\Ho[\C]$ is an honest
category is a shadow of the fact that $\C$ is meant to be equipped with a
`coherently associative composition law' that is also unital in an appropriate
sense.\footnote{Here we are being dishonest, for the definition of an
  $\infty$-category as a quasi-category does not involve such a composition
  law, at least not explicitly. But the intuition is the right one.}

\begin{remark}
  Virtually all concepts and results from category theory extend to
  $\infty$-categories, although their proofs (and the formulation of the
  concepts themselves!) might be highly non-trivial in some cases. In
  particular, one may speak of functors, natural transformations, adjunctions,
  etc. As mentioned in the introduction, we take the theory of $\infty$-categories
  as a black box from which we only extract some of its highlights.
\end{remark}

Given an $\infty$-category $\C$, there is a functor (of $\infty$-categories)
\[
  \C^\op\times\C\longrightarrow\Grpd_\infty,\qquad
  (x,y)\longmapsto\Map[\C]{x}{y},
\]
that associates to a pair of object of $\C$ its \emph{space of
  maps}.\footnote{The definition of $\infty$-category is peculiar in that the
  mapping space functor is notoriously difficult to construct, see for
  example~\cite[Section~5.8]{Cis19}.} When $\C$ is small, its transpose
\[
  \C\longrightarrow\Fun{\C^\op}{\Grpd_\infty},\qquad
  y\longmapsto\hat{y}\coloneqq\Map[\C]{-}{y}
\]
yields the \emph{Yoneda embedding}, whose target is the $\infty$-category of
functors from the \emph{opposite $\infty$-category} $\C^\op$ of $\C$ to the
$\infty$-category of $\infty$-groupoids. The Yoneda embedding is fully faithful,
which is to say that the induced maps of $\infty$-groupoids
\[
  \Map[\C]{x}{y}\stackrel{\sim}{\longrightarrow}\Map{\hat{x}}{\hat{y}},\qquad
  x,y\in\C,
\]
are \emph{invertible}, that is they are isomorphisms in the homotopy category
$\Ho(\Grpd_\infty)$, see~\cite[Proposition~5.1.3.1]{Lur09}
or~\cite[Theorem~5.8.13]{Cis19}. More generally, a functor between
$\infty$-categories $F\colon\C\to\D$ is an \emph{equivalence} if it is fully
faithful in the above sense and if the induced functor of ordinary categories
\[
  \Ho(F)\colon\Ho(\C)\to\Ho(\C),\qquad x\longmapsto F(x),
\]
is essentially surjective.

Since every set can be interpreted as a groupoid with only identity morphisms,
every ordinary category $\C$ can be interpreted as an $\infty$-category with
mapping spaces
\[
  \Map[\C]{x}{y}=\Hom[\C]{x}{y},\qquad x,y\in\C.
\]
In fact, the theory of $\infty$-categories can be regarded as an enlargement of
the theory of ordinary categories in the following sense. Let us denote by
$\cat$ the ordinary category of small categories and by $\cat[\infty]$ the
$\infty$-category of small $\infty$-categories. There is an adjunction
\[
  \begin{tikzcd}
    \Ho\colon\cat[\infty]\rar[shift left]&\cat\lar[shift left,hook']
  \end{tikzcd}
\]
in which the right adjoint $\cat\hookrightarrow\cat[\infty]$ is fully
faithful~\cite[Proposition~3.3.14]{Cis19}.\footnote{In the implementation of the
  concept of $\infty$-category in terms of simplicial sets, the fully faithful
  functor $\cat\hookrightarrow\cat[\infty]$ is modelled by the nerve functor
  $\C\mapsto N(\C)$. Since we have deliberately avoided to give a rigorous
  definition of $\infty$-category in this survey, these technicalities can
  safely be ignored. In particular, we do not distinguish between an ordinary
  category and its nerve.} The above adjunction restricts to a further
adjunction
\[
  \begin{tikzcd}
    \Ho\colon\Grpd_\infty\rar[shift left]&\Grpd\lar[shift left,hook']
  \end{tikzcd}
\]
Here, it is necessary to mention that the (small) $\infty$-groupoids span a full
($\infty$-)subcategory of $\cat[\infty]$, and that an $\infty$-category $\C$ is
an $\infty$-groupoid if and only if its homotopy category $\C$ is an ordinary
groupoid~\cite[Corollary~1.4]{Joy02}.\footnote{This is in line with the usual
  nomenclature for higher categories. In our setting, $\infty$-categories are
  models for $(\infty,1)$-categories, that is categories with morphisms of
  dimension $n$ for all ${\infty>n\geq1}$ such that the morphisms of
  dimension $n>1$ are invertible. Similarly, $\infty$-groupoids model
  $(\infty,0)$-categories, that is $(\infty,1)$-categories in which \emph{all}
  morphisms are invertible.} Another important fact is the following
characterisation of the equivalences between $\infty$-groupoids.\footnote{The
  reader should compare this with the definition of a weak homotopy equivalence
  between topological spaces.}

\begin{theorem}[{\cite[Corollary~3.8.14]{Cis19}}]
  \label{thm:gpd-equivalences}
  Let $f\colon X\to Y$ be a functor between $\infty$-groupoids. Then, $f$ is an
  equivalence if and only if the induced map
  \[
    \pi_0(f)\colon\pi_0(X)\longrightarrow\pi_0(Y)
  \]
  is bijective and, for each point $x\in X$ and each integer $n\geq1$, the
  induced map
  \[
    \pi_n(f)\colon\pi_n(X,x)\longrightarrow\pi_n(Y,f(x))
  \]
  is an isomorphism of groups.
\end{theorem}

When applying $\infty$-categorical constructions to ordinary categories, such as
the formation of functor $\infty$-categories or of slice $\infty$-categories,
one often recovers the corresponding constructions from ordinary category
theory---but not always. An important example of a construction that can take us
outside of the $1$-categorical realm (and this is a good thing) is the
following. Below, we say that a morphism in $\C$ is
\emph{invertible}\footnote{One also calls such morphisms \emph{equivalences} or
  even \emph{isomorphisms}.} if its image in $\Ho(\C)$ is an isomorphism.

\begin{defprop}[{\cite[Remark~1.3.4.2]{Lur17}, see
    also~\cite[Proposition~7.1.3]{Cis19}}]
  \label{defprop:localisation}
  Let $\C$ be an $\infty$-category and $W$ a class of morphisms in $\C$.
  Set-theoretic issues aside, there exists an $\infty$-category $\C$ and a
  functor $\gamma\colon\C\to\C[W^{-1}]$ that sends all morphisms in $W$ to
  invertible morphisms in $\C[W^{-1}]$ and that, moreover, enjoys the following
  universal property: For each $\infty$-category $\D$, restriction along
  $\gamma$ induces an equivalence of $\infty$-categories
  \[
    \gamma^*\colon\Fun{\C[W^{-1}]}{\D}\stackrel{\sim}{\longrightarrow}\Fun[W]{\C}{\D},
  \]
  where $\Fun[W]{\C}{\D}\subseteq\Fun{\C}{\D}$ denotes the full subcategory
  spanned by the functors that send all morphisms in $W$ to invertible morphisms
  in $\D$. We call the $\infty$-category $\C[W^{-1}]$ the \emph{localisation} of
  $\C$ at $W$.
\end{defprop}

\begin{remark}
  \label{rmk:localisation}
  Let $\C$ be an ordinary category and $W$ a class of maps in $\C$. Then, the
  homotopy category $\Ho(\C[W^{-1}])$ of $\C[W^{-1}]$ is canonically equivalent
  to the $1$-categorical localisation of $\C$ at $W$, see for
  example~\cite[Remark~7.1.6]{Cis19}. Beware that, in general, the canonical
  functor $\C[W^{-1}]\to\Ho(\C[W^{-1}])$ is not an equivalence, the crucial
  point being that $\C[W^{-1}]$ often has better properties than its
  $1$-categorical counterpart (compare with~\Cref{rmk:limits-Ho}). To minimise
  the potential for confusion, in this context, we refer to $\C[W^{-1}]$ as the
  $\infty$-categorical localisation of $\C$ at $W$ and denote the corresponding
  $1$-categorical localisation exclusively by $\Ho(\C[W^{-1}])$. Particularly
  important are the $\infty$-categorical localisations of (combinatorial)
  Quillen model categories and related variants, see for
  example~\cite[Appendix~A.2]{Lur09} and~\cite[Chapter~7]{Cis19}.
\end{remark}

\begin{remark}
  Setting aside set-theoretic issues again, \emph{every} $\infty$-category is
  equivalent to the $\infty$-categorical localisation of an ordinary category,
  see~\cite{BK12}.
\end{remark}

The $\infty$-categorical notions of limit and colimit play a central role in our
forthcoming discussion of stable $\infty$-categories. We recall the necessary
definitions.

\begin{definition}
  An $\infty$-groupoid $X$ is \emph{contractible} if $\pi_0(X)=*$ and
  $\pi_n(X,x)=0$ for all $n\geq1$, where $x\in X$ is an arbitary
  point.\footnote{Notice that a contractible $\infty$-groupoid is non-empty!}
\end{definition}

\begin{remark}
  Intuitively, a contractible $\infty$-groupoid has an essentially unique point,
  and this point has no non-trivial automorphisms in any dimension. According to
  \Cref{thm:gpd-equivalences}, an $\infty$-groupoid $X$ is contractible if and
  only if the map $X\to *$ is an equivalence, where $*$ is any singleton. In
  other words, a contractible $\infty$-groupoid is just as good as a singleton.
\end{remark}

\begin{definition}[{\cite[Section~1.2.12]{Lur09}}]
  Let $\C$ be an $\infty$-category. An object $*\in\C$ is \emph{final} if, for
  each object $x\in\C$, the mapping space $\Map[\C]{X}{*}$ is contractible.
  \emph{Initial} objects are defined dually, as final objects in the opposite
  $\infty$-category $\C^\op$.
\end{definition}

\begin{remark}
  \label{rmk:uniqueness-final_objects}
  Let $\C$ be an $\infty$-category. The final objects of $\C$ form an
  $\infty$-groupoid that is either empty or
  contractible~\cite[Proposition~4.4]{Joy02}. In this sense, final objects of an
  $\infty$-category are essentially unique.
\end{remark}

\begin{example}
  Any singleton defines a final object of $\Grpd_\infty$ and, in fact, the final
  objects are precisely the contractible $\infty$-groupoids. The empty set is
  the initial $\infty$-groupoid.
\end{example}

\begin{definition}[{\cite[Sections~1.2.13 and~Chapter~4]{Lur09}}]
  Let $A$ be a small $\infty$-category (for example, a small category), $\C$ an
  $\infty$-category and $F\colon A\to \C$ a functor. A \emph{limit} of $F$ is a
  final object of the $\infty$-category $\C_{/F}$ of cones\footnote{As usual, a
    cone over $F$ with apex $x\in\C$ can be identified with a natural
    transformation $\operatorname{const}_A(x)\to F$ from the constant diagram
    $A\to \C$ with value $x$, see for example~\cite[Section~6.2.1]{Cis19}.} over
  $F$. Dually, a \emph{colimit} of $F$ is an initial object of the
  $\infty$-category $\C_{F/}$ of cones under $F$.
\end{definition}

\begin{remark}
  According to \Cref{rmk:uniqueness-final_objects}, limits and colimits, if they
  exist, are essentially unique.
\end{remark}

\begin{remark}
  An $\infty$-category admits finite colimits if and only if it admits finite
  coproducts and pushouts~\cite[Corollary~4.4.2.4]{Lur09}. Similarly, if $\C$
  admits finite colimits, a functor $F\colon\C\to\D$ preserves finite colimits
  if and only if it preserves finite coproducts and
  pushouts~\cite[Corollary~4.4.2.5]{Lur09}. Passing to opposites, we obtain the
  analogous statements for finite limits.\footnote{An ordinary finite category
    need not be finite as an $\infty$-category. For example, the category
    $\Idem$ in \Cref{def:Idem} is finite as an ordinary category but not as an
    $\infty$-category, essentially because it admits strings of composable
    non-identity morphisms of arbitrary length.}
\end{remark}

\begin{remark}
  The $\infty$-category $\Grpd_\infty$ of small $\infty$-groupoids admits all
  small limits and all small colimits, see~\cite[Section~5.1.2]{Lur09}
  or~\cite[Proposition~6.2.12 and Corollary~6.3.8.]{Cis19}.
\end{remark}

\begin{remark}
  If $\C$ and $A$ are ordinary categories, then the $\infty$-categorical notions of
  $A$-shaped (co)limits in $\C$ agree with the ones of ordinary category
  theory~\cite[Section~1.2.13]{Lur09}.
\end{remark}

\begin{remark}
  The functor $\pi_0\colon\Grpd_\infty\to\Set$ preserves small
  products~\cite[\href{https://kerodon.net/tag/00HD}{Corollary~00HD}]{kerodon}.
  Consequently, if an $\infty$-category $\C$ admits small (co)products then so
  does its homotopy category $\Ho(\C)$, see
  also~\cite[Proposition~6.5.5]{Cis19}.\footnote{Beware that the converse is
    false in general.}
\end{remark}

\begin{remark}
  An important remark concerning commutative diagrams in an $\infty$-category is
  in order. In an $\infty$-groupoid $X$, it is not relevant whether two of its
  points $x$ and $y$ are equal or not, what matters is whether there is an
  invertible map ${x\stackrel{\sim}{\to} y}$ in $X$. Similarly, in an
  $\infty$-category $\C$, it is not relevant whether two maps $f,g\colon x\to y$
  in $\C$ are equal, what matters is whether there is an invertible map
  ${f\stackrel{\sim}{\to} g}$ in the $\infty$-groupoid $\Map[\C]{x}{y}$. Hence,
  when we speak of a commutative square
  \[
    \begin{tikzcd}
      w\rar{f}\dar{g}&x\dar{g'}\\
      y\rar[swap]{f'}&z
    \end{tikzcd}
  \]
  in $\C$, what is meant is not just the data depicted but---and this is
  crucial---the additional data of an identification $g'\circ f\simeq f'\circ g$
  in the $\infty$-groupoid $\Map[\C]{w}{z}$. It is not enough to ask that the
  maps $g'\circ f$ and $f'\circ g$ lie in the same connected component of
  $\Map[\C]{w}{z}$, we must really choose an identification witnessing this
  fact. In other words, in $\infty$-category theory the commutativity of a
  diagram is not a property but part of its structure. Thus, in order construct
  a well-defined diagram $f\colon A\to \C$ indexed by some small
  $\infty$-category, we must specify a potentially infinite amount of
  higher-dimensional information in a manner that is not dissimilar to the data
  required to define an $A_\infty$-functor.\footnote{As it turns out, an
    $\infty$-category $A$ is determined by its $n$-simplices $\Delta^n\to A$,
    $n\geq0$, where $\Delta^n$ denotes the path category of the quiver
    $0\to1\to2\cdots\to n$, and a functor $A\to\C$ is determined by its action
    on all the simplices of $A$, see for example~\cite{HS25a} and the references therein.} We advise the reader to keep this important
  technical subtlety in mind when pondering the commutative diagrams that appear
  later in this text.
\end{remark}

\begin{remark}
  \label{rmk:pi0-no-PB}
  The functor $\pi_0\colon\Grpd_{\infty}\to\Set$ does not preserve arbitrary
  pullbacks. Simple counterexamples arise as follows: Given an $\infty$-groupoid
  $X$ and a point $x\in X$, there is a pullback square
  \[
    \begin{tikzcd}
      \Omega(X,x)\rar\dar\ar[phantom]{dr}[description]{\text{PB}}&*\dar{x}\\
      *\rar[swap]{x}\rar& X
    \end{tikzcd}
  \]
  in the $\infty$-category $\Grpd_\infty$. However, the induced commutative
  diagram of sets
  \[
    \begin{tikzcd}
      \pi_0(\Omega(X,x))\rar\dar&*\dar\\
      *\rar&\pi_0(X)
    \end{tikzcd}
  \]
  is not a pullback diagram as soon as $\pi_0(\Omega(X,x))=\pi_1(X,x)$ has more
  than one element. This is the case, for example, if $X=BG$ is the ordinary
  groupoid with a single object $x$ whose automorphism group $\pi_1(X,x)=G$ is
  non-trivial.
\end{remark}

\begin{remark}
  \label{rmk:limits-Ho}
  Let $\C$ be an $\infty$-category. Limits in $\C$ can be characterised in terms
  of representability statements in the usual way. For example, suppose given a
  pullback diagram in $\C$ of the form
  \[
    \begin{tikzcd}
      x\times_zy\rar\dar\ar[phantom]{dr}[description]{\text{PB}}&x\dar\\
      y\rar&z.
    \end{tikzcd}
  \]
  Then, for each object $w\in\C$, the induced map of $\infty$-groupoids
  \[
    \Map[\C]{w}{x\times_z
      y}\stackrel{\sim}{\longrightarrow}\Map[\C]{w}{x}\times_{\Map[\C]{w}{z}}\Map[\C]{w}{y}
  \]
  is invertible. In particular, the induced map
  \[
    \Hom[\C]{w}{x\times_z y}=\pi_0(\Map[\C]{w}{x\times_z
      y})\stackrel{\sim}{\to}\pi_0(\Map[\C]{w}{x}\times_{\Map[\C]{w}{z}}\Map[\C]{w}{y})
  \]
  is bijective. However, as explained in \Cref{rmk:pi0-no-PB}, the canonical
  comparison map
  \[
    \begin{tikzcd}[row sep=small]
      \pi_0(\Map[\C]{w}{x}\times_{\Map[\C]{w}{z}}\Map[\C]{w}{y})\dar\\
      \pi_0(\Map[\C]{w}{x})\times_{\pi_0(\Map[\C]{w}{z})}\pi_0(\Map[\C]{w}{y})\dar[equals]\\\Hom[\C]{w}{x\times_zy}\times_{\Hom[\C]{w}{z}}\Hom[\C]{w}{y}
    \end{tikzcd}
  \]
  need not be bijective. This is the reason why, in general, the existence of
  (co)limits in an $\infty$-category, is not inherited by their homotopy
  categories (see also~\Cref{rmk:localisation}).
\end{remark}

In general, one has the following fundamental result. Below, we work in the
$\infty$-category $(\Grpd_\infty)_*=(\Grpd_{\infty})_{*/}$ of \emph{pointed
  $\infty$-groupoids}, that is $\infty$-groupoids with a prescribed basepoint.

\begin{theorem}[Serre's long exact sequence,~{\cite[Lemma~I.7.3]{GJ99}
    or~\cite[Theorem~3.8.12]{Cis19}}]
  \label{thm:Serre-les}
  Suppose given a map $f\colon (X,x_0)\to (Y,y_0)$ in the $\infty$-category of
  pointed $\infty$-groupoids. Consider a pullback diagram
  \[
    \begin{tikzcd}
      F\rar\dar\ar[phantom]{dr}[description]{\textup{PB}}&X\dar{f}\\
      *\rar[swap]{y_0}&Y
    \end{tikzcd}
  \]
  and notice that the point $x_0\in X$ defines a canonical point $\tilde{x}_0\in
  F$. Then, there is an induced long exact sequence of pointed sets (of groups
  in the range $n\geq1$)
  \[
    \begin{tikzcd}
      \cdots\rar&\pi_1(F,\tilde{x}_0)\rar&\pi_1(X,x_0)\rar&\pi_1(Y,y_0)\ar[out=-20,in=160]{dll}[description]{\partial}\\
      &\pi_0(F,\tilde{x}_0)\rar&\pi_0(X,x_0)\rar&\pi_0(Y,y_0),
    \end{tikzcd}
  \]
  where $\pi_0(X,x_0)=(\pi_0(X),[x_0])$ and similarly for the other pointed sets
  of connected components. Moreover, there is a natural right action of the
  group $\pi_1(Y,y_0)$ on the set $\pi_0(F)$ with the property that the the
  $\pi_1(Y,y_0)$-orbits are precisely the set-theoretic fibres of the induced
  map ${\pi_0(F)\to\pi_0(X)}$.\footnote{In terms of the $\pi_1(Y,y_0)$-action,
    the displayed connecting map is given by $\delta\colon [\varphi]\mapsto
    [\tilde{x_0}]\cdot[\varphi]$.}
\end{theorem}

\begin{example}
  In the case of ordinary groupoids, \Cref{thm:Serre-les} is
  elementary~\cite[Section~V.2]{GZ67}: Given a functor between pointed
  groupoids ${f\colon(X,x_0)\to(Y,y_0)}$, define the \emph{homotopy fibre} of $f$
  at $y_0$ to be the groupoid $F$ whose objects are the pairs
  \[
    (x,\varphi)=(x\in X,\ \varphi\colon y_0\stackrel{\sim}{\to} f(x)\text{ in
    }Y);
  \]
  a morphism $u\colon (x,\varphi)\to(x',\varphi)$ consists of a morphism
  $u\colon x\to x'$ in $X$ such that the following diagram in $Y$ commutes:
  \[
    \begin{tikzcd}
      y_0\dar[equals]\rar{\varphi}&f(x)\dar{f(u)}\\
      y_0\rar{\varphi'}&f(x').
    \end{tikzcd}
  \]
  Let $\tilde{x}_0\coloneqq(x_0,\id[y_0])\in F$. The construction of the exact
  sequence exact sequence
  \[
    \begin{tikzcd}
      *\rar&\pi_1(F,\tilde{x}_0)\rar&\pi_1(X,x_0)\rar&\pi_1(Y,y_0)\ar[out=-20,in=160]{dll}[description]{\partial}\\
      &\pi_0(F,\tilde{x}_0)\rar&\pi_0(X,x_0)\rar&\pi_0(Y,y_0)
    \end{tikzcd}
  \]
  is then a nice exercise that we leave to the reader.
\end{example}

\section{Lurie's stable $\infty$-categories}
\label{sec:stable-cats}

In this section we focus our attention on the class of stable
$\infty$-categories that was introduced by Lurie in his doctoral
dissertation~\cite{Lur04}; the vast majority of the results in our compilation
are due to him.

\subsection{Definition and first properties}

\begin{definition}[{\cite[Definition~1.1.1.9.]{Lur17}}]
  An $\infty$-category $\C$ is \emph{stable} if it has the following properties:
  \begin{enumerate}
  \item The $\infty$-category $\C$ is \emph{pointed}, that is there exists an
    object $0\in\C$ that is both final and initial (and thus a \emph{zero
      object}.)
  \item Let $f\colon x\to y$ be a morphism in $\C$. Then, there exists diagrams
    \[
      \begin{tikzcd}
        w\rar\dar\ar[phantom]{dr}[description]{\text{PB}}&x\dar{f}\\
        0\rar&y
      \end{tikzcd}\qquad\text{and}\qquad%
      \begin{tikzcd}
        x\rar{f}\dar\ar[phantom]{dr}[description]{\text{PO}}&y\dar\\
        0\rar&z
      \end{tikzcd}
    \]
    where the square on the left is cartesian (=pullback square) and the square
    on the right is cocartesian (=pushout square).
  \item A square in $\C$ of the form
    \[
      \begin{tikzcd}
        x\rar\dar&y\dar\\
        0\rar&z
      \end{tikzcd}
    \]
    is cartesian if and only if it is cocartesian.
  \end{enumerate}
\end{definition}

\begin{remark}
  Being stable is a \emph{property} of an $\infty$-category. This is in stark
  contrast with the situation for triangulated categories, for which both the
  suspension functor and the triangulation are additional, non-canonical
  \emph{structures}. From this perspective, stable $\infty$-categories are
  somewhat similar to abelian categories.
\end{remark}

\begin{remark}
  Stability is a self-dual property: An $\infty$-category $\C$ is stable if and
  only if its opposite $\C^\op$ is stable.
\end{remark}

The following basic result, whose analogue for triangulated categories is false
in general, is an immediate consequence of the fact that limits and colimits in functor
$\infty$-categories are computed pointwise~\cite[Corollary~5.1.2.3]{Lur09}.

\begin{proposition}[{\cite[Proposition~1.1.3.1.]{Lur17}}]
  \label{prop:stable_functor_categories}
  Let $A$ be a small $\infty$-category and $\C$ a stable $\infty$-category.
  Then, the $\infty$-category $\Fun{A}{\C}$ of functors $A\to\C$ is stable.
\end{proposition}

Before going any further, we explain how stable $\infty$-categories relate to
triangulated categories. We begin by introducing the necessary terminology.

\begin{defprop}
  Let $\C$ be a pointed $\infty$-category such that, for each object $x\in\C$,
  there exist squares
  \[
    \begin{tikzcd}
      x\rar\dar\ar[phantom]{dr}[description]{\text{PO}}&0\dar\\
      0\rar&\Sigma(x)
    \end{tikzcd}\qquad\text{and}\qquad%
    \begin{tikzcd}
      \Omega(x)\rar\dar\ar[phantom]{dr}[description]{\text{PB}}&0\dar\\
      0\rar&x
    \end{tikzcd}
  \]
  where the square on the left is cocartesian and the square on the right is
  cartesian. We call $\Sigma(x)$ the \emph{suspension} of $x$ and $\Omega(x)$
  the \emph{loops} of $x$. The constructions $x\mapsto\Sigma(x)$ and
  $x\mapsto\Omega(x)$ yield a pair of adjoint functors\footnote{This can be
    shown easily using the pointwise formulas for left and right Kan extensions,
    see~\cite[Section~4.3.2]{Lur09} and~\cite[Remark~1.1.1.7]{Lur17} for details.}
  \[
    \begin{tikzcd}
      \Sigma\colon\C\rar[shift left]&\C\noloc\Omega\lar[shift left],
    \end{tikzcd}
  \]
  where $\Sigma$ is left adjoint to $\Omega$.
\end{defprop}

\begin{remark}
  An ordinary category $\C$ is stable if and only if $\C$ is the zero category.
  Indeed, by definition, for every object $x\in\C$ in a stable $\infty$-category
  there are canonical identifications
  \[
    x\simeq\Omega(\Sigma(x))\qquad\text{and}\qquad x\simeq\Sigma(\Omega(x)),
  \]
  see also~\Cref{thm:equivalent-def-stable}. If $\C$ is an ordinary category, we
  clearly have $\Sigma(x)=0$ and $\Omega(x)=0$, so that $x=0$ as well. This
  shows that stability is genuinely an $\infty$-categorical notion.
\end{remark}

\begin{defprop}
  Let $\C$ be a stable $\infty$-category. We call a bicartesian\footnote{That
    is, a square that is both cartesian and cocartesian.} square of the form
  \[
    \begin{tikzcd}
      x\rar{f}\dar\ar[phantom]{dr}[description]{\square}&y\dar{g}\\
      0\rar&z
    \end{tikzcd}
  \]
  a \emph{fibre-cofibre sequence}, where the small square in the middle is simply a decoration. The
  object $x$ is called the \emph{fibre} of $g$ and is usually denoted $\fib[g]$;
  dually, the object $z$ is called the \emph{cofibre} of $f$ and is usually
  denoted $\cofib[f]$. The passage from a morphism to its (co)fibre is
  functorial: There are functors\footnote{These functors are exact in the sense
    of \Cref{defprop:exact_functors} since, by construction, they admit
    adjoints, see also~\cite[Proposition~5.2.3.5]{Lur09} and~\cite[Remar~1.1.1.8]{Lur17}.}
  \[
    \fib\colon\Fun{s\to t}{\C}\longrightarrow\C,\qquad f\mapsto\fib[f],
  \]
  and
  \[
    \cofib\colon\Fun{s\to t}{\C}\longrightarrow\C,\qquad f\mapsto\cofib[f].
  \]
  By considering the extended diagram\footnote{By Pasting Lemma for bicartesian
    squares~\cite[Lemma~4.4.2.1]{Lur09}, the outer rectangle is also
    bicartesian. Therefore, the object in the bottom-right corner can be
    canonically identified with $\Sigma(x)$.}
  \[
    \begin{tikzcd}
      x\rar{f}\dar\ar[phantom]{dr}[description]{\square}&y\dar{g}\rar\ar[phantom]{dr}[description]{\square}&0\dar\\
      0\rar&z\rar&\Sigma(x)
    \end{tikzcd}
  \]
  in which the square on the right is also bicartesian, we obtain canonical
  identifications
  \[
    \Sigma(\fib[g])\simeq\cofib[g]\qquad\text{and}\qquad\fib[g]\simeq\Omega(\cofib[g]).
  \]
  These identifications are shadows of the fundamental adjoint equivalences of
  stable $\infty$-categories
  \[
    \begin{tikzcd}
      \operatorname{Cofib}\colon\Fun{s\to t}{\C}\rar[shift left]{\sim}&\Fun{s\to
        t}{\C}\noloc\operatorname{Fib}\lar[shift left]{\sim}
    \end{tikzcd}
  \]
  whose action on objects is apparent; for example,
  \[
    \operatorname{Cofib}(f\colon x\to y)=(y\to\cofib{f}).
  \]
\end{defprop}

\begin{remark}
  Fibre-cofibre sequences in stable $\infty$-categories play the role of
  triangles in triangulated categories,
  see~\Cref{thm:triangulated_homotopy_categories}.
\end{remark}

\begin{definition}[{\cite[Definition~C.1.5.1]{Lur18SAG}}]
  An $\infty$-category $\C$ is \emph{semi-additive} if it has the following
  properties:
  \begin{enumerate}
  \item The $\infty$-category $\C$ is pointed.
  \item The $\infty$-category $\C$ admits finite products and finite coproducts.
  \item For every pair of objects $x,y\in\C$, the canonical map
    \[
      \left(\begin{smallmatrix} \id[x]&0\\0&\id[y]
      \end{smallmatrix}\right)\colon x\amalg y\longrightarrow x\times y
  \]
  is invertible.
\end{enumerate}
We say that $\C$ is \emph{additive} if it is semi-additive and the following
condition holds:
\begin{enumerate}
  \setcounter{enumi}{3}
\item For each object $x\in\C$, the map
  \[
    \left(\begin{smallmatrix}\id[x]&\id[x]\\
      0&\id[x]
    \end{smallmatrix}\right)\colon x\amalg x\longrightarrow x\times x
\]
is invertible.
\end{enumerate}
In particular, the homotopy category of a (semi-)additive $\infty$-category is a
(semi-)additive category in the usual sense.
\end{definition}

The following result shows that stable $\infty$-categories indeed `enhance'
triangulated categories.

\begin{theorem}[{\cite[Theorem~1.1.2.14]{Lur17}}]
  \label{thm:triangulated_homotopy_categories}
  Let $\C$ be a stable $\infty$-category. The following statements hold:
  \begin{enumerate}
  \item The $\infty$-category $\C$ is additive, and hence so is its homotopy
    category $\Ho(\C)$.
  \item The pair $(\Ho(\C),\Sigma)$ admits the structure of a triangulated
    category, whose class of triangles is given by the closure under
    isomorphisms of the image in $\Ho(\C)$ of all diagrams in $\C$ of the form
    \[
      \begin{tikzcd}
        x\rar\dar\ar[phantom]{dr}[description]{\square}&y\rar\dar\ar[phantom]{dr}[description]{\square}&0\dar\\
        0\rar&z\rar&\Sigma(x)
      \end{tikzcd}
    \]
    where both squares are bicartesian.
  \end{enumerate}
\end{theorem}

\begin{remark}
  All triangulated categories that arise `naturally' in mathematics can be
  realised as the triangulated homotopy category of a stable $\infty$-category; this includes
  algebraic triangulated categories in the sense of Keller~\cite{Kel06} and
  topological triangulated categories in the sense of Schwede~\cite{Sch10a}. The
  case of algebraic triangulated categories is dealt with
  in~\Cref{prop:Cisinski-Frobenius-stable-cat}.
\end{remark}

\begin{remark}
  There are strong results on uniquness of $\infty$-categorical enhancements of
  triangulated categories~\cite{Ant18,CNS22} with important precursors
  in~\cite{Sch01,SS02,Sch07}.
\end{remark}

We record the apparent definition of `stable subcategory' of a stable
$\infty$-category for later use.

\begin{defprop}[{\cite[Lemma~1.1.3.3]{Lur17}}]
  Let $\C$ be a stable $\infty$-category and $\A\subseteq\C$ a full subcategory.
  The following statements are equivalent:
  \begin{enumerate}
  \item The subcategory $\A\subseteq\C$ contains a zero object and is closed
    under the formation fibres and cofibres.
  \item The subcategory $\A\subseteq\C$ is closed under the formation of
    cofibres and under the action of the suspension functor of $\C$ and its
    inverse.
  \end{enumerate}
  If the above equivalent conditions are satisfied, we say that $\A$ is a
  \emph{stable subcategory} of $\C$.
\end{defprop}

Stable $\infty$-categories admit several useful equivalent characterisations.

\begin{theorem}[{\cite[Proposition~1.1.3.4, Corollary~1.4.2.27]{Lur17}}]
  \label{thm:equivalent-def-stable}
  Let $\C$ be a pointed $\infty$-category. The following statements are
  equivalent:
  \begin{enumerate}
  \item The $\infty$-category $\C$ is stable.
  \item The $\infty$-category $\C$ admits finite limits,\footnote{An $\infty$-category is \emph{finite} if it contained in the smallest
      subcategory of $\cat[\infty]$ that contains $\emptyset$, $*$ and
      $[1]=\set{0<1}$, and which is closed under pushouts. A (co)limit is \emph{finite} if it is indexed by a finite
    $\infty$-category.} finite colimits, and
    the classes of cartesian and cocartesian squares in $\C$ coincide.
  \item The $\infty$-category $\C$ admits finite limits, finite colimits, and
    the adjoint functors
    \[
      \begin{tikzcd}
        \Sigma\colon\C\rar[shift left]&\C\noloc\Omega\lar[shift left]
      \end{tikzcd}
    \]
    are mutually inverse equivalences of $\infty$-categories.
  \item The $\infty$-category $\C$ admits finite limits and the functor
    \[
      \Omega\colon\C\longrightarrow\C,\qquad x\longmapsto 0\times_x0
    \]
    is an equivalence.
  \item The $\infty$-category $\C$ admits finite colimits and the functor
    \[
      \Sigma\colon\C\longrightarrow\C,\qquad x\longmapsto 0\amalg_x0
    \]
    is an equivalence.
  \end{enumerate}
\end{theorem}

We now turn our attention to functors between stable $\infty$-categories. The
relevant notion is the following.

\begin{defprop}[{\cite[Proposition~1.1.4.1, Corollary~1.4.2.14]{Lur17}}]
  \label{defprop:exact_functors}
  Let $F\colon\C\to\D$ be a functor between stable $\infty$-categories. The
  following conditions are equivalent:
  \begin{enumerate}
  \item The functor $F$ preserves zero objects and fiber-cofiber sequences.
  \item The functor $F$ preserves zero objects and the canonical natural
    transformation
    \[
      \Sigma_{\D}\circ F\Longrightarrow F\circ\Sigma_{\C}
    \]
    is invertible.\footnote{The beauty of this condition cannot be overstated.}
  \item The functor $F$ is \emph{left exact}, that is $F$ preserves finite
    limits.
  \item The functor $F$ is \emph{right exact}, that is $F$ preserves finite
    colimits.
  \end{enumerate}
  If the above equivalent conditions hold, we say that $F$ is an \emph{exact}
  functor.
\end{defprop}

\begin{remark}
  It follows from \Cref{defprop:exact_functors} that equivalences between stable
  $\infty$-categories are exact.
\end{remark}

\begin{definition}
  We denote by $\cat[\infty]<ex>$ the $\infty$-category of essentially small
  stable $\infty$-categories and exact functors between them. This is a non-full
  (!) subcategory of the $\infty$-category $\cat[\infty]$ of essentially small
  $\infty$-categories.
\end{definition}

\begin{remark}
  In view of \Cref{thm:triangulated_homotopy_categories} and
  \Cref{defprop:exact_functors}, there is a well-defined functor
  $\C\mapsto\Ho(\C)$ from $\cat[\infty]<ex>$ to the ordinary category of
  triangulated categories and triangle functors between them.
\end{remark}

\Cref{prop:stable_functor_categories} admits the following variant, which is
also generally false for triangulated categories.

\begin{proposition}[{\cite[Proposition~1.4.2.16]{Lur17}}]
  Let $\C$ and $\D$ be essentially small stable $\infty$-categories. Then, the
  $\infty$-category $\Fun<ex>{\C}{\D}$ of exact functors $\C\to D$ is a stable
  $\infty$-category.
\end{proposition}

The following criterion for detecting equivalences between stable
$\infty$-categories is useful.

\begin{proposition}[{\cite[Theorem~7.6.10]{Cis19}}]
  \label{prop:equivalences_of_stable_categories}
  Let $F\colon\C\to\D$ be an exact functor between stable $\infty$-categories.
  Then, $F$ is an equivalence if and only if the triangle functor
  $\Ho(F)\colon\Ho(\C)\to\Ho(\D)$ is an equivalence of ordinary categories.
\end{proposition}

\begin{remark}
  It is worth mentioning that~\cite[Theorem~7.6.10]{Cis19} treats the more
  general case of $\infty$-categories with finite limits and left-exact
  functors. The proof in the case of stable $\infty$-categories is much easier, for
  $\operatorname{Hom}$-functors in triangulated categories are (co)homological
  and there are canonical isomorphisms of abelian groups\footnote{Here, we have
    taken the homotopy groups of the mapping space $\Map[\C]{x}{y}$ based at the
  null map $x\to y$. Owing to the fact that
  $\Map[\C]{x}{y}\simeq\Omega(\Map[\C]{x}{\Sigma(y)},0)$ is a loop space and
  hence an $H$-group, its
  connected components are pairwise equivalent and
  its homotopy groups are therefore independent of the choice of basepoint (see also~\Cref{rmk:mapping-spectra}).}
  \[
    \pi_n(\Map[\C]{x}{y},0)\cong\Hom[\C]{\Sigma^n(x)}{y},\qquad x,y\in\C,\
    n\geq0.
  \]
  We invite the reader to give a proof of
  \Cref{prop:equivalences_of_stable_categories} with this in mind.
\end{remark}

Stable $\infty$-categories are amenable to robust gluing techniques, as the
following result makes manifest.

\begin{theorem}[{\cite[Theorem~1.1.4.4, Proposition~1.1.4.6]{Lur17}}]
  The $\infty$-category $\cat[\infty]<ex>$ admits small limits and small
  filtered colimits, and these are preserved by the inclusion
  ${\cat[\infty]<ex>\subset\cat[\infty]}$.
\end{theorem}

\begin{remark}
  In fact, the $\infty$-category $\cat[\infty]<ex>$ admits small
  colimits~\cite[Corollary~4.25]{BGT13}, but these are in general not preserved
  by the inclusion ${\cat[\infty]<ex>\subset\cat[\infty]}$ unless these are
  filtered. For a simple counterexample, notice that the coproduct of two stable
  $\infty$-categories $\C$ and $\D$, as $\infty$-categories, is their disjoint
  union, which no longer has a zero object.\footnote{The reader should compare
    the situation with the forgetful functor from the category of abelian groups
    to the category of sets, which also preserves small limits and small
    filtered colimits, but not arbitrary small colimits (such as small
    coproducts).} Important examples of non-filtered colimits of stable
  $\infty$-categories are Verdier quotients,
  see~\Cref{defprop:Verdier_quotient}.
\end{remark}

The construction of Verdier quotients lifts to the $\infty$-categorical world as
follows.

\begin{defprop}[{\cite[Theorem~I.3.3.]{NS18}}]
  \label{defprop:Verdier_quotient}
  Let $F\colon \C\hookrightarrow\D$ be a fully faithful exact functor between
  stable $\infty$-categories.\footnote{For example, the inclusion of a stable
    subcategory.} The \emph{Verdier quotient} of $\D$ by the essential image of
  $\C$, denoted $\D/\C$, is the $\infty$-categorical localisation of $\D$ at the
  class of morphism whose cofibre lies in the essential image of $\C$. The
  Verdier quotient $\D/\C$ is a stable $\infty$-category and there exists a
  cocartesian diagram in the $\infty$-category $\cat[\infty]<ex>$ of the
  form\footnote{It is worth mentioning that, in this particular case, the square
  itself is essentially unique. Indeed, it correspond to the datum of a trivialisation
  $p\circ F\simeq 0$ and, since the $0\in\Fun<ex>{\C}{\D/\C}$ is a zero object,
  such a trivialisation is therefore essentially unique, compare
  with~\cite[Definition~5.8]{BGT13} and~\cite[Remark~3.2]{Dyc21}.}
  \[
    \begin{tikzcd}
      \C\rar[hookrightarrow]{F}\dar\ar[phantom]{dr}[description]{\text{PO}}&\D\dar{p}\\
      0\rar&\D/\C
    \end{tikzcd}
  \]
  In particular, the canonical functor $p\colon\D\to\D/\C$ is exact and its
  restriction along $F$ vanishes. Furthermore, for each stable $\infty$-category
  $\E$, restriction along $p$ induces an equivalence of stable
  $\infty$-categories
  \[
    p^*\colon\Fun<ex>{\D/\C}{\E}\stackrel{\sim}{\longrightarrow}\Fun[\C]<ex>{\D}{\E},
  \]
  where $\Fun[\C]<ex>{\D}{\E}$ denotes the $\infty$-category of exact functors
  $\D\to\E$ whose restriction along $F\colon \C\hookrightarrow\D$ vanishes.
\end{defprop}

\begin{remark}
  Let $F\colon \C\to\D$ be a fully faithful exact functor between stable
  $\infty$-categories. It follows immediately from the universal property of
  localisations and of the homotopy category construction that there is a
  canonical equivalence of categories
  \[
    \Ho(\D)/\Ho(\C)\stackrel{\sim}{\longrightarrow}\Ho(\D/\C),
  \]
  where the left-hand side denotes the usual Verdier quotient of the
  triangulated category $\Ho(\D)$ by the essential image of $\Ho(\C)$, compare
  with~\cite[Remark~7.1.6]{Cis19}.
\end{remark}

\subsection{Idempotent-complete stable $\infty$-categories}

\begin{definition}
  \label{def:Idem}
  Let $\Idem$ be the category with a single object $*$ and a single non-identity
  morphism $e\colon *\to *$, so that $e^2=e$. An \emph{idempotent endomorphism}
  in an $\infty$-category $\C$ is a functor
  \[
    \Idem\longrightarrow\C,\quad *\longmapsto x,
  \]
  see~\cite[\href{https://kerodon.net/tag/03ZT}{Definition
    03ZT}]{kerodon}.\footnote{Given an object $x\in\C$, specifying an idempotent
    endomorphism $e\colon x\to x$ requires one to provide an infinite amount of
    data, namely, the data of compatible identifications
    \[
      \underbrace{e\circ\cdots\circ e}_{n\text{ times}}\simeq e,\qquad n\geq1,
    \]
    compare with~\cite[Section~4]{GHW21}.} Similary, let $\Ret$ be the category
  with two objects $x$ and $y$ and whose only non-idendity morphisms are
  $i\colon y\to x$, $p\colon x\to y$ and $e\coloneq i\circ p\colon x\to x$.
  Notice that $\Idem$ identifies with the full subcategory of $\Ret$ spanned by
  the single object $x$. An idempotent endomorphism $\Idem\to\C$ is \emph{split}
  if it admits an extension
  \[
    \begin{tikzcd}
      &\Ret\dar[dotted]\\
      \Idem\rar\urar[hookrightarrow]&\C,
    \end{tikzcd}
  \]
  see~\cite[\href{https://kerodon.net/tag/03ZX}{Definition 03ZX}]{kerodon}. An
  $\infty$-category is \emph{idempotent complete} if every idempotent morphism
  in $\C$ is
  split~\cite[\href{https://kerodon.net/tag/0406}{Definition~0406}]{kerodon}.
\end{definition}

The following result is an $\infty$-categorical variant of
\cite[Theorem~1.5]{BS01}.

\begin{proposition}[{\cite[Propositions~5.1.4.2 and~5.1.4.9]{Lur09} and~\cite[Corollary~1.1.3.7]{Lur17}}]
  Every $\infty$-category admits $\C$ an idempotent completion, that is a fully
  faithful functor $i\colon\C\hookrightarrow\C^\flat$ into an idempotent
  complete $\infty$-category $\C^\flat$ such that every object of $\C^\flat$ is
  a retract of an object in the essential image of $\C$. Moreover, for each
  idempotent-complete $\infty$-category $\D$, restriction along $i$ induces an
  equivalence of $\infty$-categories
  \[
    i^*\colon\Fun{\C^\flat}{\D}\stackrel{\sim}{\longrightarrow}\Fun{\C}{\D}.
  \]
  If $\C$ and $\D$ are stable, then $\C^\flat$ is stable, the functor
  $i\colon\C\to\C^\flat$ is exact, and there is an induced equivalence of stable
  $\infty$-categories
  \[
    i^*\colon\Fun<ex>{\C^\flat}{\D}\stackrel{\sim}{\longrightarrow}\Fun<ex>{\C}{\D}.
  \]
\end{proposition}

In the case of stable $\infty$-categories, idempotent-completeness can be
checked a the level of the homotopy category.

\begin{proposition}[{\cite[Corollary~1.1.3.7, Lemma 1.2.4.6]{Lur17}}]
  Let $\C$ be a stable $\infty$-category. The following statements are
  equivalent:\footnote{Beware that these statements are not equivalent for
    arbitrary $\infty$-categories.}
  \begin{enumerate}
  \item The $\infty$-category $\C$ is idempotent complete.
  \item The homotopy category $\Ho(\C)$ is idempotent complete.
  \end{enumerate}
\end{proposition}

\subsection{Stable $\infty$-categories from Frobenius exact categories}

Recall that a Quillen exact category is said to be \emph{Frobenius} if it has
enough projectives, enough injectives, and the classes of projective and
injective objects coincide~\cite[Section~I.2]{Hap88}. The \emph{stable category}
of $\E$ is the quotient $\underline{\E}$ of $\E$ by its ideal of morphisms that factor through a
projective-injective object; this is a triangulated category by Happel's
Theorem~\cite[Theorem~I.2.6]{Hap88}. By definition, a \emph{stable isomorphism}
is a morphism in $\E$ whose class in $\underline{\E}$ is an isomorphism. The
following theorem, which shows that all algebraic triangulated categories admit
a stable $\infty$-categorical enhancement, is well-known to experts.

\begin{proposition}[Cisinski]
  \label{prop:Cisinski-Frobenius-stable-cat}
  Let $\E$ be a Frobenius exact category and $W$ the class of stable
  isomorphisms in $\E$. Then, the $\infty$-categorical localisation $\E[W^{-1}]$
  is a stable $\infty$-category and, moreover, there is a canonical equivalence
  of triangulated categories
  \[
    \Ho(\E[W^{-1}])\simeq\underline{\E}.
  \]
  Furthermore, the localisation functor $\gamma\colon\E\to\E[W^{-1}]$ is exact
  in the sense that it sends admissible short exact sequences in $\E$ to
  fibre-cofibre sequences in $\E[W^{-1}]$.
\end{proposition}
\begin{proof}
  The argument below is essentially the one given
  in~\cite[Proposition~4.19.]{Cis10}, where it is shown that a Frobenius exact
  category gives rise to a stable/triangulated derivator. We use results
  from~\cite[Chapter~7]{Cis19} that can be understood as $\infty$-categorical
  refinements of results in~\cite{Cis10}. The homotopy category
  $\Ho(\E[W^{-1}])$ is the $1$-categorical localisation of $\E$ at the class of
  stable isomorphisms~\cite[Remark~7.1.6]{Cis19}, which is well-kown to be
  equivalent to the stable category $\underline{\E}$, see for
  example~\cite[Lemma~2.2.2]{Kra22}. The claims in the proposition follow as
  soon as we show that $\E[W^{-1}]$ is a stable $\infty$-category. Firstly,
  since $\E$ has a zero object, so does the $\infty$-category $\E[W^{-1}]$,
  see~\cite[Remark~7.1.15]{Cis19}. Secondly, it is readily verified that $\E$ is
  endowed with the structure of a category of fibrant objects in the sense
  of~\cite[Definition~7.5.7]{Cis19}, with fibrations the admissible epimorphisms
  (=deflations) and weak equivalences given by the stable isomorphisms
  and~\cite[Proposition~7.5.6]{Cis19} tells us that the $\infty$-category
  $\E[W^{-1}]$ has finite limits. Dually, the $\infty$-category $\E[W^{-1}]$ has
  finite colimits. It remains to show that the loop functor
  $\Omega\colon\E[W^{-1}]\longrightarrow\E[W^{-1}]$ is an equivalence (\Cref{thm:equivalent-def-stable}).
  According to~\cite[Theorem~7.6.10]{Cis19}, this can
  be verified at the level of homotopy categories (recall that $\Omega$ is a right
  adjoint), where we know already that
  the claim holds. The agreement between the triangulated structures on
  $\Ho(\E[W^{-1}])$ and on $\underline{\E}$ follows by direct inspection, using
  explicit description of pushouts and pullbacks in $\E[W^{-1}]$ given in the
  proof of \cite[Proposition~7.5.6]{Cis19}, keeping in mind the explicit
  description of the triangles in $\underline{\E}$ given
  in~\cite[Theorem~I.2.6]{Hap88}. The claim on the exactness of the localisation
  functor is also shown in~\cite[Proposition~7.5.6]{Cis19}.
\end{proof}

\begin{remark}
  Let $\E$ be an essentially small Frobenius exact category. Implicitly, it is
  shown in~\cite[Theorem~A.19]{AGH19} that the $\infty$-categorical enhancement
  of Buchweitz's Theorem~\cite{Buc21} holds: there is a canonical equivalence of
  stable $\infty$-categories
  \[
    \E[W^{-1}]\stackrel{\sim}{\longrightarrow}\DerCat[b]{\E}/\DerCat[b]{\operatorname{proj}(\E)},
  \]
  where the right-hand side denotes the canonical $\infty$-categorical
  enhancement of the (small) singularity category of $\E$ via the Verdier
  quotient construction of~\Cref{defprop:Verdier_quotient}.
\end{remark}

\begin{remark}
  Recall that every algebraic triangulated category can be enhanced to a
  pre-triangulated differential graded category~\cite[Section~4.3]{Kel94}. Each
  differential graded category gives rise to an $\infty$-category by applying
  Lurie's differential graded nerve construction~\cite[Section~1.3.1]{Lur17},
  and the resulting $\infty$-category is stable if the original differential
  graded category is pre-triangulated~\cite[Theorem~3.18]{Fao17b}; working over
  a field, similar results hold for $A_\infty$-categories. For the case of
  stable differential graded categories, see~\cite[Section~6]{Che26}.
\end{remark}

\subsection{Compactly-generated and presentable stable $\infty$-categories}

There is also a robust theory of large stable $\infty$-categories, at least
when these can be controlled by small $\infty$-categories in a suitable sense. We recall the
necessary definitions.

\begin{definition}
  \label{def:cgen}
  Let $\C$ be a stable $\infty$-category with small coproducts, so that the
  triangulated homotopy category $\Ho(\C)$ also admits small coproducts. An object
  $x\in\C$ is \emph{compact} if it is compact as an object of the triangulated homotopy
  category $\Ho(\C)$; that is if the functor
  \[
    \Hom[\Ho(\C)]{x}{-}\colon\Ho(\C)\longrightarrow\operatorname{Ab}
  \]
  preserves small coproducts.
  Similarly, we say that $\C$ is \emph{compactly generated}
  if its homotopy category $\Ho(\C)$ is a compactly generated triangulated
  category, that is there exist a set $X$ of objects of $\C$ such that the
  smallest triangulated subcategory of $\Ho(\C)$ containing $X$ and which is
  closed small coproducts is $\Ho(\C)$ itself.
\end{definition}

\begin{remark}
  Let $\C$ be an $\infty$-category with small filtered colimits. An object
  $x\in\C$ is compact if the corepresentable functor
  \[
    \Map[\C]{x}{-}\colon\C\longrightarrow\Grpd_\infty
  \]
  preserves small filtered colimits. Moreover, a compactly generated 
  stable $\infty$-category is compactly-generated in the sense of~\Cref{def:cgen} 
  if and only if it is compactly-generated in the sense of~\cite[Definition~5.5.7.1]{Lur09}, that is
  the full subcategory $\C^\omega\subseteq\C$ of compact objects is essentially
  small and the canonical functor
  \[
    \operatorname{Ind}(\C^\omega)\stackrel{\sim}{\longrightarrow}\C
  \]
  is an equivalence~\cite[Proposition~1.4.4.1, Corollary~1.4.4.2,
  and~Remark~1.4.4.3]{Lur17}. Here, $\operatorname{Ind}(\C^\omega)$ denotes the
  $\operatorname{Ind}$-completion of $\C^\omega$, which is the $\infty$-category
  obtained from $\C$ by freely adjoining small filtered
  colimits~\cite[Section~5.3.5]{Lur09}.
\end{remark}

\begin{example}
  \label{ex:spectra}
  The most important stable $\infty$-category is, arguably, the
  stable $\infty$-category $\Spectra$ of spectra, which can be defined as the limit of the
  tower~\cite[Definition~1.4.3.1, Remark~1.4.2.25]{Lur17}
  \[
    \begin{tikzcd}
      \cdots\rar{\Omega}&\rar{\Omega}(\Grpd_\infty)_*&\rar{\Omega}(\Grpd_\infty)_*&\rar{\Omega}(\Grpd_\infty)_*&(\Grpd_\infty)_*,
    \end{tikzcd}
  \]
  where $(\Grpd_\infty)_*$ denotes the $\infty$-category of pointed $\infty$-groupoids.
  The $\infty$-category $\Spectra$ is compactly generated by a canonical compact
  object $\SS\in\Spectra$ known as the \emph{sphere spectrum}.  The role of the $\infty$-category of spectra in the theory of
  stable $\infty$-categories is entirely analogous to that of the category of
  abelian groups in the theory of abelian categories.
\end{example}

\begin{defprop}[{\cite[Proposition~1.4.4.9]{Lur17}}]
  A stable $\infty$-category $\C$ is \emph{presentable} if there exists a small
  $\infty$-category $A$ such that $\C$ is an accessible left-exact
  localisation\footnote{That is, there exists a reflective localisation
    \[
      \begin{tikzcd}[ampersand replacement=\&]
        \Fun{A}{\Spectra}\rar[shift left]{L}\&\C\lar[shift left,hook']
      \end{tikzcd}
    \]
    such that the localisation functor $L$ preserves finite limits and such that $\C$ is
    closed under the formation of $\kappa$-filtered colimits in $\Fun{A}{\Spectra}$
    for some regular cardinal $\kappa$.} of
  $\Fun{A}{\Spectra}$. In particular, the stable $\infty$-category $\Spectra$ of spectra
  is presentable.
\end{defprop}

\begin{remark}
  Compactly-generated stable $\infty$-categories are
  presentable.
\end{remark}

The following result  shows that presentable stable $\infty$-categories can be
regarded as enhancements of Neeman's well-generated triangulated categories.

\begin{theorem}[\cite{Ros05,Lur09}]
  Let $\C$ be a presentable stable $\infty$-category. Then, its triangulated
  homotopy category $\Ho(\C)$ is well generated in the sense of Neeman~\cite{Nee01}.
\end{theorem}
\begin{proof}
  By \cite[Proposition~6.10]{Ros05}, the homotopy category of a combinatorial
  stable model category is well generated. Since every presentable
  $\infty$-category is the underlying $\infty$-category of a combinatorial model
  category~\cite[Proposition~A.3.7.6]{Lur09} (see
  also~\cite[Proposition~1.3.4.22]{Lur17}), the claim follows.
\end{proof}

\begin{definition}
  We denote by $\PrLSt$ the $\infty$-category of presentable stable
  $\infty$-ca\-te\-go\-ries and colimit-preserving functors between them. Given stable $\infty$-categories
  $\C,\D\in\PrLSt$, we denote by $\LFun{\C}{\D}$ the $\infty$-category of
  colimit-preserving functors $\C\to\D$. By the Adjoint Functor
  Theorem~\cite[Corollary~5.5.2.9]{Lur09}, every colimit-preserving functor
  $\C\to\D$ between presentable $\infty$-categories is a left adjoint, hence the
  notation.
\end{definition}

The following theorem provides some evidence for the richness of the theory of
presentable stable $\infty$-categories. For differential graded variants of this
result in the compactly-generated case, see~\cite[Corollary 7.6]{Toe07}
and~\cite[Remark~2.23]{Jas25}; the well-generated case is treated
in~\cite{LRG22}.

\begin{theorem}[{\cite[Corollary~4.8.2.19]{Lur17} and~\cite[Lemma~D.5.3.3.]{Lur18SAG}}]
  \label{thm:PrLSt-tensor_product}
  The $\infty$-ca\-te\-go\-ry $\PrLSt$ admits a closed symmetric monoidal structure
  with the internal $\operatorname{Hom}$ functor $(\C,\D)\mapsto\LFun{\C}{\D}$.
  In particular, for each pair of presentable stable $\infty$-categories $\C$
  and $\D$, there exists a presentable stable $\infty$-category $\C\otimes\D$
  such that, for each presentable stable $\infty$-category $\E$, there is a
  canonical equivalence of $\infty$-categories
  \[
    \LFun{\C\otimes\D}{\E}\stackrel{\sim}{\longrightarrow}\LFun{\C}{\LFun{\D}{\E}}.
  \]
  The unit of this monoidal structure is the $\infty$-category $\Spectra$ of
  spectra. Moreover, compactly generated stable $\infty$-categories form a
  monoidal subcategory of $\PrLSt$.
\end{theorem}

\Cref{thm:PrLSt-tensor_product} is behind several interesting formulas, such as
the following very general form of the Eilenberg--Watts
Theorem~\cite{Eil60,Wat60}.

\begin{theorem}[{\cite[Proposition~7.1.2.4]{Lur17}
    and~\cite[Proposition~D.7.2.3]{Lur18SAG}}]
  Let $\C$ and $\D$ be presentable stable $\infty$-categories and suppose that
  $\C$ is compactly generated.\footnote{More generally, one may assume that $\C$
  is `compactly assembled', see~\Cref{rmk:dualisability}.} Then, there is an equivalence of presentable
  stable $\infty$-categories
  \[
    \D\otimes\C^\vee\stackrel{\sim}{\longrightarrow}\LFun{\C}{\D},
  \]
  where $\C^\vee\coloneqq\LFun{\C}{\Spectra}$ is the dual of $\C$. If $R$ and
  $S$ are ring spectra, there are equivalences of
  compactly-generated stable $\infty$-categories
  \[
    \DerCat{R}\otimes\DerCat{S^\op}\stackrel{\sim}{\longrightarrow}\LFun{\DerCat{S}}{\DerCat{R}};
  \]
  moreover, $\DerCat{R}\otimes\DerCat{S^\op}$ is equivalent to the stable
  $\infty$-category of $S$-$R$-bimodule spectra.\footnote{Implicitly, we work
    with right modules and write $\DerCat{R}$ for the derived $\infty$-category
    of $R$-module spectra, which we do not define here, and that is simply
    called the $\infty$-category of $R$-module spectra in~\cite{Lur17}.} In
  particular $\DerCat{S}^\vee$ can be identified with $\DerCat{S^{\op}}$.
\end{theorem}

\begin{remark}
  Applications of Lurie's tensor product to the study of Holm-J{\o}r\-gen\-sen's
  $Q$-shaped derived categories~\cite{HJ22} can be found
  in~\cite[Remark.~3.23]{Jas25}.
\end{remark}

\subsection{Smooth and proper stable $\infty$-categories}

The category $\operatorname{Hmo}$ of differential graded categories up to Morita
equivalence~\cite{Toe07}, constructed using~\cite[Théorème~5.3]{Tab05}, plays an
important role in the theory of algebraic triangulated categories. Below we
discuss its analogue in the context of stable
$\infty$-categories.

\begin{definition}
  We denote by $\cat[\infty]<ex,\omega>$ the full subcategory of
  $\cat[\infty]<ex>$ spanned by the stable $\infty$-categories that are
  idempotent complete.
\end{definition}

\begin{theorem}[{\cite[Theorem~3.1]{BGT13}}]
  The $\infty$-category $\cat[\infty]<ex,\omega>$ admits a closed symmetric
  monoidal structure with the internal $\operatorname{Hom}$ functor
  $(\C,\D)\mapsto\Fun<ex>{\C}{\D}$. In particular, for each pair of
  idempotent-complete stable $\infty$-categories $\C$ and $\D$, there exists an
  idempotent-complete stable $\infty$-category $\C\wotimes\D$ such that, for
  each idempotent complete stable $\infty$-category $\E$, there is a canonical
  equivalence of $\infty$-categories
  \[
    \Fun<ex>{\C\wotimes\D}{\E}\stackrel{\sim}{\longrightarrow}\Fun<ex>{\C}{\Fun<ex>{\D}{\E}}.
  \]
  The unit of this monoidal structure is the $\infty$-category $\Spectra^\omega$
  of compact spectra (see~\Cref{ex:spectra}).
\end{theorem}

\begin{remark}
  The $\infty$-category $\cat[\infty]<ex,\omega>$ is the underlying
  $\infty$-category of a combinatorial model
  structure~\cite[Theorem~4.23]{BGT13} and hence it is a presentable
  $\infty$-category. In fact, more is true: The $\infty$-category
  $\cat[\infty]<ex,\omega>$ is compactly generated in the sense
  of~\cite[Definition~5.5.7.1]{Lur09}, see~\cite[Corollary~4.25]{BGT13}.
\end{remark}

The notions of smooth and proper differential graded categories admit variants for
stable $\infty$-categories. In order to state the definition we need a
preparatory result.

\begin{defprop}[{\cite[Remark~5.3.2]{Lur09} and~\cite[Proposition~3.2]{BGT13}}]
  Let $\C$ be an essentially small stable $\infty$-category. There is a
  canonical equivalence of $\infty$-categories
  \[
    \operatorname{Ind}(\C)\stackrel{\sim}{\longrightarrow}\Fun<ex>{\C^\op}{\Spectra}.
  \]
  In particular, the Yoneda embedding $\C\to\Fun{\C^\op}{\Grpd_\infty}$
  factors through the \emph{stable Yoneda embedding}
  \[
    \C\longrightarrow\Fun<ex>{\C^\op}{\Spectra}
  \]
  whose transpose defines the \emph{diagonal $\C$-bimodule}
  \[
    \C^\op\wotimes\C\longrightarrow\Spectra,\qquad (x,y)\longmapsto\MapSp[\C]{x}{y},
  \]
  whose values are called \emph{mapping spectra}.
\end{defprop}
\begin{proof}
  We give a simple proof of the equivalence
 \[
   \operatorname{Ind}(\C)\stackrel{\sim}{\longrightarrow}\Fun<ex>{\C^\op}{\Spectra}
 \]
 that we expect should be known to experts. The takeaways are the role of the tensor
 product of presentable stable $\infty$-categories and the universal property of
 the $\operatorname{Ind}$-completion.
 \begin{align*}
   \operatorname{Ind}(\C)&\simeq\Spectra\otimes\operatorname{Ind}(\C)&\text{\Cref{thm:PrLSt-tensor_product}}\\
                         &\simeq\RFun{\Spectra^\op}{\operatorname{Ind}(\C)}&\text{\cite[Proposition~4.8.1.17]{Lur17}}\\
                         &\simeq\LFun{\operatorname{Ind}(\C)}{\Spectra^\op}^\op&\text{\cite[Proposition~5.2.6.2]{Lur09}}\\
                         &\simeq\Fun<ex>{\C}{\Spectra^\op}^\op&\text{\cite[Proposition~5.3.5.10]{Lur09}}\\
                         &\simeq\Fun<ex>{\C^\op}{\Spectra}.
 \end{align*}
 Here, $\RFun{\Spectra^\op}{\operatorname{Ind}(\C)}$ denotes the
 $\infty$-category of functors $\Spectra^\op\to\operatorname{Ind}(\C)$ that are
 right adjoints. That the Yoneda embedding $\C^\op\to\Fun{\C^\op}{\Spaces}$
 factors through the $\operatorname{Ind}$-completion of $\C$ is
 clear~\cite[Remark~5.3.5.2]{Lur09}.
\end{proof}

\begin{remark}
  \label{rmk:mapping-spectra}
  Let $\C$ be a stable $\infty$-category and $x,y\in\C$ a pair of objects. The
  existence of functorial mapping spectra is anticipated by the fact that the sequence
  of mapping spaces
  \[
    \Map[\C]{x}{\Sigma^n(y)},\qquad n\geq0,
  \]
  which, in view of the canonical identifications
  \[
    \Map[\C]{x}{\Sigma^{n-1}(y)}\simeq\Omega(\Map[\C]{x}{\Sigma^n(y)},0),
  \]
  define the spectrum $\MapSp[\C]{x}{y}\in\C$, compare with~\Cref{ex:spectra}.
  From the general formula,\footnote{See for
    example~\cite[Remark~3.9]{Gre07}.} for computing the stable homotopy groups of a spectrum
  (see~\Cref{ex:spectra-silting}) it follows that there are isomorphisms of
  abelian groups
  \[
    \pi_n^s(\MapSp[\C]{x}{y})\cong\pi_n(\Map[\C]{x}{y},0)\cong\Hom[\C]{\Sigma^n(x)}{y},\qquad
    n\geq0,
  \]
  and
  \[
    \pi_{-n}^s(\MapSp[\C]{x}{y})\cong\pi_0(\Map[\C]{x}{\Sigma^n(y)})=\Hom[\C]{x}{\Sigma^n(y)},\qquad
    n\geq0.
  \]
  The takeaway here is that negative stable homotopy groups of mapping spectra correspond to
  positive extensions.
\end{remark}

\begin{definition}[Kontsevich]
  Let $\C$ be an essentially small stable $\infty$-category.
  \begin{enumerate}
  \item We say that $\C$ is \emph{proper} if for each pair of objects $x,y\in\C$
    the mapping spectrum $\MapSp[\C]{x}{y}$ is compact as an object of $\Spectra$.
  \item We say that $\C$ is \emph{smooth} if the diagonal $\C$-bimodule
    is a compact object of the stable $\infty$-category
    $\Fun<ex>{\C^\op\wotimes\C}{\Spectra}$.
  \end{enumerate}
\end{definition}

The following is an $\infty$-categorical analogue of \cite[Theorem~5.8]{CT12},
which treats the case of differential graded categories. It gives a conceptual
characterisation of the smooth and proper stable $\infty$-categories that are
idempotent complete. Recall that an object $V$ of a closed symmetric monoidal
($\infty$-)category $\category{V}$ is dualisable if there exist an object $V^*$ and morphisms
\[
  \operatorname{ev}\colon V^*\otimes V\longrightarrow
  \mathbf{1}\qquad\text{and}\qquad\operatorname{coev}\colon\mathbf{1}\longrightarrow
  V\otimes V^*,
\]
called \emph{evaluation} and \emph{coevaluation}, as well as commutative
triangles
\[
  \begin{tikzcd}
    V^*\simeq
    V^*\otimes\mathbf{1}\rar{\mathbf{1}\otimes\operatorname{ev}}\drar[swap]{\id}&V^*\otimes
    V\otimes V^*\dar{\operatorname{ev}\otimes\mathbf{1}}\\
    &\mathbf{1}\otimes V^*\simeq V^*
  \end{tikzcd}\qquad\text{and}\qquad
  \begin{tikzcd}
    V\simeq\mathbf{1}\otimes V\rar{\operatorname{coev}\otimes\id}\drar[swap]{\id}&V\otimes V^*\otimes V\dar{\id\otimes\operatorname{ev}}\\
    &V\otimes\mathbf{1}\simeq V
  \end{tikzcd}
\]
In this case, the right adjoint $W\mapsto\underline{\category{V}}(V,W)$ of the functor
$-\otimes V\colon\category{V}\to\category{V}$ canonically identifies with the
functor $W\mapsto W\otimes V^*$. For example, the dualisable objects in the
category of vector spaces over a field are precisely those of finite dimension.

\begin{theorem}[{\cite[Theorem~3.7]{BGT13}}]
  \label{thm:dualisable-Hmo}
  The dualisable objects of the symmetric mo\-noi\-dal $\infty$-category
  $\cat[\infty]<ex,\omega>$ are precisely the idempotent-complete stable $\infty$-categories that are
  smooth and proper. Moreover, the dual of a dualisable object $\D$ is $\D^\op$,
  so that for each idempotent-complete stable $\infty$-category $\C$ there is a
  canonical equivalence of $\infty$-categories
  \[
    \C\wotimes\D^\op\stackrel{\simeq}{\longrightarrow}\Fun<ex>{\D}{\C}
  \]
  and one has an equivalence $\D^\op\simeq\Fun<ex>{\D}{\Spectra^\omega}$ in this case.
\end{theorem}

\begin{remark}
  \label{rmk:dualisability}
  In contrast with the restrictive characterisation in~\Cref{thm:dualisable-Hmo}, all compactly-generated stable
  $\infty$-categories are dualisable as objects of the symmetric monoidal
  $\infty$-category $\PrLSt$, see~\cite[Proposition~D.2.7.3]{Lur18SAG} and
  compare with the proof of \cite[Theorem 5.8]{CT12}.\footnote{The subtle but
    crucial point is that a morphism $F\colon\C\to\D$ in $\PrLSt$ between
    compactly-generated stable $\infty$-categories, that is a colimit-preserving
    functor, is not required to preserve compact objects.} The dualisable
  objects of $\PrLSt$ can be characterised~\cite[Proposition~D.7.3.1]{Lur18SAG}
  and have attracted a significant amount of interest, not least due to their
  central role in Efimov's continuous $K$-theory~\cite{Efi24}, see
  also~\cite{Ram24,Efi25a}.
\end{remark}

\subsection{$t$-structures on stable $\infty$-categories and realisation
  functors}

The theory of $t$-structures on triangulated categories~\cite{BBDG18} extends
largely without changes to stable $\infty$-categories. We recall only a minimal
amount of the theory and refer the reader to~\cite[Section~1.2.1]{Lur17} for
more details where, as below, homological (!) indexing convention is
used.

\begin{definition}[{\cite[Definition~1.2.1.1]{Lur17}}]
  Let $\C$ be a stable $\infty$-category. A \emph{$t$-struc\-ture} on $\C$ is a
  pair of full subcategories $t=(\C_{\geq0},\C_{\leq0})$ that induce a
  $t$-structure on the triangulated homotopy category $\Ho(\C)$ in the sense
  of~\cite{BBDG18}. In particular, the \emph{heart}
  \[
    \C^\heartsuit\coloneqq \C_{\geq0}\cap\C_{\leq0}
  \]
  is an ordinary abelian category and we have the familiar adjoint pairs
  \[
    \begin{tikzcd}
      \C_{\geq0}\rar[hookrightarrow,shift left]&\C\lar[shift left]{\tau_{\geq0}}
    \end{tikzcd}\qquad\text{and}\qquad
    \begin{tikzcd}
      \C\rar[shift left]{\tau_{\leq0}}&\C_{\leq0}\lar[shift left,hook']
    \end{tikzcd}
  \]
  as well as the homological functor
  \[
    \pi_0^t\coloneqq\tau_{\geq0}\circ\tau_{\leq0}\colon\C\longrightarrow\C^\heartsuit,
  \]
  etc.
\end{definition}

An important novelty in the theory is the construction of completions of
$t$-stru\-tures, for which enhancements seem to be crucial.

\begin{defprop}[{\cite[Proposition~1.2.1.17]{Lur17}}]
  Let $\C$ be a stable $\infty$-category equipped with a $t$-structure
  $(\C_{\geq0},\C_{\leq0})$. The \emph{left completion} $\widehat{\C}$ of $\C$
  (with respect to the given $t$-structure) is the limit of the diagram
  \[
    \begin{tikzcd}
      \cdots\rar{\tau_{\leq 2}}&\C_{\leq 2}\rar{\tau_{\leq1}}&\C_{\leq
        1}\rar{\tau_{\leq0}}&\C_{\leq 0}\rar{\tau_{\leq-1}}&\C_{\leq
        -1}\rar{\tau_{\leq-2}}&\C_{\leq-2}\rar{\tau_{\leq-3}}&\cdots
    \end{tikzcd}
  \]
  The $\infty$-category $\widehat{\C}$ is stable and is endowed with a
  $t$-structure $(\widehat{\C}_{\geq0},\widehat{\C}_{\leq0})$ such that the
  exact canonical functor $\C\to\widehat{\C}$ restricts to an equivalence of
  $\infty$-categories
  $\C_{\leq0}\stackrel{\sim}{\longrightarrow}\widehat{\C}_{\leq0}$.
\end{defprop}

\begin{defprop}[{\cite[Proposition~1.2.1.19]{Lur17}, \cite[Lemma~2.15]{Ant18}}]
  \label{defprop:left-complete}
  Let $\C$ be a stable $\infty$-category equipped with a $t$-structure
  $(\C_{\geq0},\C_{\leq0})$. We say that the $t$-structure on $\C$ is \emph{left
    complete} if the canonical exact functor $\C\to\widehat{\C}$ is an
  equivalence. The notion of \emph{right-complete} $t$-structure is defined
  dually. Consider now the following conditions:
  \begin{enumerate}
  \item\label{eq:left-separated} The $t$-structure on $\C$ is left complete.
  \item\label{eq:left-nondeg}  The $t$-structure on $\C$ is left non-degenerate, that is
    \[
      \bigcap_{n\in\ZZ}\C_{\geq n}\simeq\set{0}.
    \]
  \end{enumerate}
  Then \eqref{eq:left-separated}$\Rightarrow$\eqref{eq:left-nondeg} and, if $\C$ admits countable products and the aisle
  $\C_{\geq0}$ is closed under these, then the two conditions are equivalent.
\end{defprop}

\begin{remark}
  Examples of $t$-structures on unbounded derived ($\infty$-)categories that are
  not left complete can be found in~\cite{Nee11}. In addition
  to~\cite{Lur17,Lur18SAG}, other useful sources of information on left-complete
  $t$-structures on stable $\infty$-categories and pre-triangulated differential
  graded category include~\cite{Ant18,CNS22}. In
  representation theory, left-complete $t$-structures have been used
  in~\cite{PV18}, for example.
\end{remark}

\begin{remark}
  Krause~\cite{Kra20a} and Neeman~\cite{Nee20} have developed alternative
  approaches to completing triangulated categories with striking
  applications~\cite{Nee18,BCR24}.
\end{remark}

The results in~\cite[Appendix~C]{Lur18SAG} suggest that the following class of
$t$-structures is of particular interest; we give them a name to
simplify the exposition.

\begin{defprop}
  \label{def:Grothendieck_t-structure}
  Let $\C$ presentable stable $\infty$-category. We say that a $t$-structure
  $(\C_{\geq0},\C_{\leq0})$ on $\C$ is a \emph{Grothendieck
    $t$-structure} if it is right-complete and its coaisle $\C_{\leq0}$ is closed under the formation of
  filtered colimits in $\C$.\footnote{$t$-structures satisfying the condition
    that the coaisle is closed under filtered colimits are also called
    `homotopically smashing,' see for example~\cite{SSV23}.} Under these conditions, the following statements hold:
  \begin{enumerate}
  \item \cite[Proposition~1.4.4.13]{Lur17} The aisle $\C_{\geq0}$ and the
    coaisle $\C_{\leq0}$ are presentable
    $\infty$-categories.
  \item \cite[Proposition~C.1.4.1]{Lur18SAG} The aisle $\C_{\geq0}$ is a Grothendieck prestable $\infty$-category in
    the sense of~\cite[Definition~C.1.4.2]{Lur18SAG}. In particular, filtered
    colimits in $\C_{\geq0}$ commute with finite limits.
  \item \cite[Remark~1.3.5.23]{Lur18SAG} The heart $\C^\heartsuit$ is a Grothendieck
    category.
  \end{enumerate}
  Moreover, $\C$ is equivalent to the
  $\infty$-category $\Spectra(\C_{\geq0})$ of \emph{spectrum objects in $\C$},
  which we may identify with the limit of the tower~\cite[Proposition~1.4.2.24]{Lur17}
  \[
    \begin{tikzcd}
      \cdots\rar{\Omega}&\rar{\Omega}\C_{\geq0}&\rar{\Omega}\C_{\geq0}&\rar{\Omega}\C_{\geq0}&\C_{\geq0}.
    \end{tikzcd}
  \]
\end{defprop}



The following universal property of unbounded derived $\infty$-categories
explains the relevance of Grothendieck $t$-structures. This universal property
should be understood as a vast extension of the construction of realisation
functors~\cite{Bei87,BBDG18}, see~\Cref{ex:realisation_unbounded}.

\begin{definition}
  \label{def:DerCatG}
  Let $\G$ be a Grothendieck category. The \emph{derived $\infty$-category} of
  $\G$ is the $\infty$-categorical localisation
  \[
    \DerCat{\G}\coloneqq\Ch{\G}[\mathrm{qis}^{-1}]
  \]
  of the category of chain complexes in $\G$ at the class of quasi-isomorphisms.
\end{definition}

\begin{theorem}[{\cite[Propositions~1.3.5.9 and~1.3.5.21]{Lur17},
    \cite[Proposition~C.3.1.1 and~C.3.2.1, Theorem~C.5.4.9]{Lur18SAG}}]
  \label{thm:univ-prop-DerCatG}
  Let $\G$ be a Grothendieck category. The following statements hold:
  \begin{enumerate}
  \item The derived $\infty$-category $\DerCat{\G}$ is a presentable stable
    $\infty$-category. It is equipped with the standard $t$-structure with aisle
    \[
      \DerCat{\G}_{\geq0}\coloneqq\set{X\in\DerCat{\G}}[\forall n<0,\ H_n(X)=0]
    \]
    and the coaisle
    \[
      \DerCat{\G}_{\leq0}\coloneqq\set{X\in\DerCat{\G}}[\forall n>0,\ H_n(X)=0].
    \]
    This is a non-degenerate Grothendieck $t$-structure\footnote{The standard
      $t$-structure on $\DerCat{\G}$ need not be left complete.
      \Cref{defprop:left-complete} cannot be applied since the standard aisle
      need not be closed under products (unless products in $\G$ are exact).}
     whose heart $\DerCat{\G}^\heartsuit$ is equivalent to $\G$.
   \item Let $\C$ be a presentable stable $\infty$-category equipped with a
     non-degenerate Grothendieck $t$-structure $(\C_{\geq0},\C_{\leq0})$. Then,
     restriction to the hearts induces an equivalence of $\infty$-categories
    \[
      \LFun<\text{$t$-}ex>{\DerCat{\G}}{\C}\stackrel{\sim}{\longrightarrow}\LFun<ex>{\G}{\C^\heartsuit};
    \]
    here, $\LFun<\text{$t$-}ex>{\DerCat{\G}}{\C}$ denotes the $\infty$-category
    of colimit-preserving $t$-exact functors $\DerCat{\G}\to\C$, and
    $\LFun<ex>{\G}{\C^\heartsuit}$ denotes the ordinary category of
    co\-li\-mit-preserving exact functors $\G\to\C^\heartsuit$.
  \end{enumerate}
\end{theorem}

\begin{remark}
  \Cref{thm:univ-prop-DerCatG} admits variants that describe universal
  properties of corresponding variants of the unbounded derived category. For
  example, a universal property of the bounded-below derived $\infty$-category
  $\DerCat[-]{\A}$ of an essentially small abelian category $\A$ with enough
  projectives is given in~\cite[Theorem~1.3.3.2]{Lur17}.
\end{remark}

\begin{example}
  \label{ex:realisation_unbounded}
  Let $\C$ be a presentable stable $\infty$-category equipped with an Groth\-en\-dieck
  non-degenerate $t$-structure $(\C_{\geq0},\C_{\leq0})$.
  \Cref{thm:univ-prop-DerCatG} applied to the Groth\-en\-dieck category $\C^\heartsuit$
  implies that restriction to the hearts induces an equivalence
  \[
    \LFun<\text{$t$-}ex>{\DerCat{\C^\heartsuit}}{\C}\stackrel{\sim}{\longrightarrow}\LFun<ex>{\C^\heartsuit}{\C^\heartsuit}.
  \]
  Choosing a pre-image of the identity functor of $\C^\heartsuit$, we obtain an
  $\infty$-categorical realisation functor
  \[
    \operatorname{real}_t\colon\DerCat{\C^\heartsuit}\longrightarrow\C.
  \]
\end{example}

\begin{example}
  \label{ex:spectra-silting}
  By construction, the sphere spectrum $\SS\in\Spectra$ is a compact silting
  object in the sense of~\cite[Definition~4.1]{AI12}. Therefore, the
  Hoshino--Kato--Miyachi Theorem~\cite[Theorem~1.3]{HKM02} implies that
  $\Spectra$ is endowed with a non-degenerate $t$-structure, called the
  \emph{Postnikov $t$-structure}, with aisle
  \[
    \Spectra_{\geq0}\coloneqq\set{X\in\Spectra}[\forall n>0,\
    \Hom[\Spectra]{\SS}{\Sigma^n(X)}=0]
  \]
  and coaisle
  \[
    \Spectra_{\leq0}\coloneqq\set{X\in\Spectra}[\forall n<0,\
    \Hom[\Spectra]{\SS}{\Sigma^n(X)}=0].
  \]
  By \Cref{defprop:left-complete} and its dual, using again that the sphere
  spectrum is compact, this $t$-structure is left and right complete. Moreover,
  the endomorphism algebra of $\SS$ in $\Ho(\Spectra)$ is isomorphic to the ring
  of integer numbers, and hence the heart $\Spectra^\heartsuit$ of this
  $t$-structure is equivalent to the category $\operatorname{Ab}$ of abelian
  groups. The abelian groups
  \[
    \pi_n^s(X)\coloneqq\Hom[\Spectra]{\Sigma^n(\SS)}{X},\qquad n\in\ZZ,
  \]
  are called the \emph{stable homotopy groups} of the spectrum $X$. In
  particular, the stable homotopy groups $\pi_n^s(\SS)$, $n\geq0$, are the stable
  homotopy groups of spheres~\cite{IWX20}. The
  corresponding realisation functor
  \[
    \DerCat{\operatorname{Ab}}\longrightarrow\Spectra
  \]
  associates to a chain complex of abelian groups its so-called
  Eilenberg--Mac Lane spectrum. These statements hold, more generally, for any
  compactly-generated stable $\infty$-category $\C$ that is generated by a
  compact silting object $S\in\C$, in which case the heart $\C^{\heartsuit}$ is
  equivalent to the category of right modules over the ordinary ring
  $\Hom[\C]{S}{S}$. Furthermore, by the Schwede--Shipley Recognition
  Theorem~\cite[Theorem~7.1.2.1]{Lur17}, the $\infty$-category $\C$ is
  equivalent to the derived $\infty$-category of the ring spectrum of
  endomorphisms of $S$. In particular, if $T\in\C$ is a compact tilting object,
  then $\C$ is equivalent to the derived $\infty$-category of the ordinary ring
  $\Hom[\C]{T}{T}$ and the corresponding $t$-structure is the standard
  $t$-structure.
\end{example}

\begin{remark}
  \label{rmk:complicial}
  Lurie introduced the following notion of `complexity' of a
  $t$-structure \cite[Definition~C.5.3.1]{Lur18SAG}: Given an integer $n\geq$, a
  $t$-structure $(\C_{\geq0},\C_{\leq0})$ on a stable $\infty$-category $\C$ is
  \emph{$n$-complicial} if for every object of the aisle $X\in\C_{\geq0}$ there
  exist an object
  \[
    \overline{X}\in\C_{\geq0}\cap\C_{\leq n}
  \]
  and a morphism $\overline{X}\to X$ such that the induced map
  $\pi_0(\overline{X})\to\pi_0(X)$ is an epimorphisms in the heart. For example,
  it is easy to show that the standard $t$-structure on the derived
  ($\infty$-)category of an abelian category is
  $0$-complicial~\cite[C.~5.3.2]{Lur18SAG}, and the standard $t$-structure on
  the derived ($\infty$-)category of a connective differential graded algebra
  $A$ is $n$-complicial if and only if $H_*(A)$ is concentrated in degrees $0$
  to $n$ inclusive~\cite[Example~C.5.3.5]{Lur18SAG}. This notion is
  closely related to the novel notions of Grothendieck
  $n$-category~\cite[Section~C.5.4]{Lur18SAG} and of abelian
  $n$-categories~\cite{Ste23,Moc25}. In these terms, one has the following
  fundamental characetrisations of the derived $\infty$-category of a
  Grothendieck category $\G$ and some of its variants:\footnote{It is worth
    noting that the coaisles of the canonical $t$-structures mentioned below are
    equivalent to each other as $\infty$-categories.}
  \begin{itemize}
  \item \cite[Remark~C.5.4.11]{Lur18SAG} The unbounded derived $\infty$-category
    $\DerCat{\G}$ is universal among presentable stable $\infty$-categories
    equipped with a non-de\-ge\-ne\-ra\-te $0$-complicial Groth\-en\-dieck
    $t$-structure with heart $\G$.
  \item \cite[Proposition~C.5.20 and~Theorem C.5.8.8]{Lur18SAG} The unbounded
    homotopy $\infty$-ca\-te\-go\-ry $\operatorname{K}(\operatorname{Inj}(\G))$ if
    complexes of injective objects in $\G$ is universal among presentable stable
    $\infty$-categories equipped with a  (left-)anticomplete\footnote{For a convenient
      equivalent characterisation of this property, see~\cite[Corollary
      C.5.5.10]{Lur18SAG}.} $0$-complicial Grothendieck $t$-structure with heart
    $\G$. For this reason, Lurie calls $\operatorname{K}(\operatorname{Inj}(\G))$
    the \emph{anticomplete derived $\infty$-category} of $\G$.
\item \cite[Corollary~C.5.9.7]{Lur18SAG} There is a presentable stable
  $\infty$-category $\DerCat*{\G}$, called the \emph{completed derived
    $\infty$-category of $\G$}, that is universal among presentable stable
  $\infty$-categories equipped with a left-complete weakly
  $0$-com\-pli\-cial\footnote{See~\cite[Definition~C.5.5.13]{Lur18SAG}.}
  Grothendieck $t$-structure with heart $\G$.
\end{itemize}
The above characterisations of $\DerCat{\G}$ and
$\operatorname{K}(\operatorname{Inj}(\G))$ are used to established
uniqueness-of-enhancements results in~\cite{Ant18}, see also~\cite{CNS22} for
further such results proven using differential graded enhancements. Finally, we
mention that $n$-com\-pli\-cial $t$-structures arising from simple-minded
collections have been investigated recently in~\cite{Plo26}.
\end{remark}

\subsection{$t$-injective objects and cosilting objects}

In this subsection we fix a presentable stable $\infty$-category $\C$ equipped
with a Grothendieck $t$-structure $(\C_{\geq0},\C_{\leq0})$ in the sense of
\Cref{def:Grothendieck_t-structure}.

\begin{defprop}[{\cite[Definition~C.5.7.2, Proposition~C.5.7.3]{Lur18SAG}}]
  \label{defprop:injectives}
  For an object $Q\in\C$, the following statements are equivalent:
  \begin{enumerate}
  \item The object $Q$ lies in the coaisle $\C_{\leq0}$ and, for each object
    $X\in\C_{\leq0}$ of the coaisle, the abelian group $\Hom[\C]{X}{\Sigma(Q)}$
    vanishes.
  \item For every object $X\in\C$, the canonical morphism of abelian groups
    \[
      \Hom[\C]{X}{Q}\stackrel{\sim}{\longrightarrow}\Hom[\C^\heartsuit]{\pi_0^t(X)}{\pi_0^t(Q)}
    \]
    is an isomorphism.
  \end{enumerate}
  If these equivalent conditions hold, we say that $Q$ is a \emph{$t$-injective
    object}\footnote{Lurie calls such objects simply `injective,' but we prefer
    to emphasise the dependence on the $t$-structure.} of $\C$.
\end{defprop}

\begin{remark}
  In the differential graded context, $t$-injective objects are investigated for
  example in~\cite{Sha18,GLVdB21,GRG23,SS23} under the names `injective
  differential graded
  modules,' `derived injective objects,' or `$\operatorname{Ext}$-injective
  objects of the coaisle.' It is also worth mentioning that many results on
  $t$-injective objects hold in much greater generality, often without assuming
  the existence of any enhancement, see for example~\cite{SS23}.
\end{remark}

The following result should be compared with~\cite[Theorem~0.2]{Sha18}, that
treats the case of injective differential graded modules. The case of (partial)
cosilting $t$-structures is treated in~\cite[Lemma~2.8]{AHMV17}
and~\cite[Lemma~5.5(2)]{Lak20}.

\begin{theorem}[{\cite[Theorem~C.5.7.4]{Lur17}}]
  \label{thm:injectives-coasile}
  Let $\operatorname{Inj}(\C)$ denote the $\infty$-category of $t$-injective
  objects of $\C$, and $\operatorname{Inj}(\C^\heartsuit)$ the ordinary category
  of injective objects of the Grothendieck category $\C^\heartsuit$. Then, the
  functor
  \[
    \pi_0^t\colon\Ho(\operatorname{Inj}(\C))\stackrel{\sim}{\longrightarrow}\operatorname{Inj}(\C^\heartsuit)
  \]
  is an equivalence of ordinary categories. In particular, each injective object
  of the Grothendieck category $\C^\heartsuit$ lifts to an essentially unique
  $t$-injective object of $\C$.
\end{theorem}

\begin{remark}
  In terms of $t$-injective objects, one has the following simpler
  characterisations of anticomplete and completed derived $\infty$-categories of
  a Grothendieck category $\G$ (compare with~\Cref{rmk:complicial}):
   \begin{itemize}
  \item \cite[Proposition~C.5.20 and~Theorem C.5.8.8]{Lur18SAG} The unbounded
    homotopy $\infty$-ca\-te\-go\-ry $\operatorname{K}(\operatorname{Inj}(\G))$ is
  universal among presentable stable $\infty$-categories equipped with an
  anticomplete Grothendieck
  $t$-structure with heart $\G$ such that every $t$-injective object lies in the heart.
\item \cite[Corollary~C.5.9.7]{Lur18SAG} There completed derived
  $\infty$-category $\DerCat*{\G}$, called the \emph{completed derived
    $\infty$-category of $\G$}, that is universal among presentable stable
  $\infty$-categories equipped with a left-complete Grothendieck $t$-structure
  with heart $\G$ such that every $t$-injective object lies in the heart.
  \end{itemize}
\end{remark}

Recall that a Grothendieck category $\G$ admits an injective cogenerator, that
is an injective object $Q\in\G$ such that, for each object $X\in\G$,
\[
  \Hom[\G]{X}{E}=0\qquad\Longleftrightarrow\qquad X=0,
\]
see for example~\cite[Theorem~9.6.3]{KS06}. The following observation relates
$t$-injective objects to cosilting objects.

\begin{proposition}
  \label{prop:cosiltings}
  Suppose that the Grothendieck $t$-structure on $\C$ is also left
  non-degenerate. Let $E\in\C$ be a $t$-injective object such that
  $\pi_0^t(E)\in\C^\heartsuit$ is an injective cogenerator. Then, there are
  equalities
  \begin{align*}
    \C_{\geq0}&=\set{X\in\C}[\forall n<0,\ \Hom[\C]{X}{\Sigma^n(E)}=0]\intertext{and}
    \C_{\leq0}&=\set{X\in\C}[\forall n>0,\ \Hom[\C]{X}{\Sigma^n(E)}=0].
  \end{align*}
  In other words, $E$ is a cosilting object of $\C$ in the sense
  of~\cite{PV18,NSZ19}. Moreover, if the stable $\infty$-category $\C$ is
  compactly generated, then $E$ is a pure-injective object of $\C$ in the sense
  of~\cite[Definition~1.1]{Kra00}.
\end{proposition}
\begin{proof}
  The argument is well known to experts. We include it for the convenience of
  the reader. First, observe that, for $X\in\C$ and $n\in\ZZ$,
  \Cref{defprop:injectives} yields isomorphisms of abelian groups
  \begin{align*}
    \Hom[\C^\heartsuit]{\pi_n^t(X)}{\pi_0^t(E)}&\cong\Hom[\C^\heartsuit]{\pi_0^t(\Omega^n(X))}{\pi_0^t(E)}\\
                                             &\cong\Hom[\C]{\Omega^n(X)}{E}\\
                                             &\cong\Hom[\C]{X}{\Sigma^n(E)}.
  \end{align*}
  Since $\pi_0^t(E)$ is an injective cogenerator of $\C^\heartsuit$, we conclude that
  \[
    \Hom[\C]{X}{\Sigma^n(E)}=0\quad\Leftrightarrow\quad\Hom[\C^\heartsuit]{\pi_n^t(X)}{\pi_0^t(E)}=0\quad\Leftrightarrow\quad\pi_n^t(X)=0.
  \]
  Hence, since the $t$-structure on $\C$ is non-degenerate by assumption, there
  are equalities of subcategories~(see~\cite[Proposition~1.3.7]{BBDG18}
  or~\cite[Theorem~IV.4.11]{GM03})
  \begin{align*}
    \C_{\geq0}&=\set{X\in\C}[\forall n<0,\ \pi_n^t(X)=0]\\
              &=\set{X\in\C}[\forall n<0,\ \Hom[\C]{X}{\Sigma^n(E)}=0]\intertext{and}
                \C_{\leq0}&=\set{X\in\C}[\forall n>0,\ \pi_n^t(X)=0]\\
              &=\set{X\in\C}[\forall n>0,\ \Hom[\C]{X}{\Sigma^n(E)}=0].
  \end{align*}
  When $\C$ is compactly generated, the claim on the pure-injectivity of $E$ is
  shown in~\cite[Theorem~3.6]{AHMV17}.
\end{proof}

\begin{remark}
  Suppose that the stable $\infty$-category $\C$ is compactly generated. In this
  case, it follows from~\cite[Corollary 3.8]{AHMV17}
  and~\cite[Theorem~4.6]{Lak20} that the non-degenerate Grothendieck
  $t$-structures on $\C$ are precisely those that are determined by a
  pure-injective cosilting object of $\C$ as in \Cref{prop:cosiltings}.
\end{remark}

\Cref{prop:cosiltings} provides a rich supply of cosilting objects of
topological origin. Here we only point out the following important example.

\begin{example}[{\cite[Example~C.5.7.10]{Lur18SAG}}]
  The (compact silting) $t$-structure on the stable $\infty$-category $\Spectra$
  of spectra that is induced by the sphere spectrum (\Cref{ex:spectra-silting})
  satisfies the assumptions in \Cref{prop:cosiltings}. The $t$-injective object
  $I_{\QQ/\ZZ}\in\Spectra$ lifting the injective cogenerator
  $\QQ/\ZZ\in\operatorname{Ab}\simeq\Spectra^\heartsuit$ is called the
  \emph{Brown--Comenetz dual of the sphere spectrum}. By \Cref{prop:cosiltings},
  $I_{\QQ/\ZZ}$ is a cosilting object in $\Spectra$. This example generalises to
  any $t$-structure generated by a compact silting object in a (necessarily
  compactly-generated) stable $\infty$-category, see also~\cite[Examples~6.8]{SS23}.
\end{example}

\subsection{The universal property of the bounded derived $\infty$-category}

Bounded derived categories enjoy a
pleasant universal property when promoted to stable $\infty$-categories.

\begin{definition}
  Let $\A$ be an essentially small abelian category. The \emph{bounded derived
    $\infty$-category} of $\A$ is the $\infty$-categorical localisation
  \[
    \DerCat[b]{\A}\coloneqq\Ch[b]{\A}[\mathrm{qis}^{-1}]
  \]
  of the category of bounded chain complexes in $\A$ at the class of quasi-isomorphisms.
\end{definition}

\begin{remark}
  Let $\A$ be an essentially small abelian category. The bounded derived
  $\infty$-category $\DerCat[b]{\A}$ is stable. Indeed, $\DerCat[b]{\A}$ can be
  equivalently constructed as the Verdier quotient of
  $\operatorname{K}^{\mathrm{b}}(\A)$ by its stable subcategory of acyclic
  complexes, where
  \[
    \operatorname{K}^{\mathrm{b}}(\A)\coloneqq\Ch[b]{\A}[\mathrm{heq}^{-1}]
  \]
  denotes the $\infty$-categorical localisation of $\Ch[b]{\A}$ at the class of
  homotopy equivalences~\cite[Section~7.4]{BCKW25}. Notice that
  $\operatorname{K}^{\mathrm{b}}(\A)$ is a stable $\infty$-category
  (\Cref{prop:Cisinski-Frobenius-stable-cat}).
\end{remark}

\begin{theorem}[{\cite[Corollary~7.4.12]{BCKW25}}]
  \label{thm:uniprop-Db}
  Let $\A$ be an essentially small abelian or, more generally, Quillen exact
  category. Then, for each stable $\infty$-category $\C$, restriction along the
  canonical functor $\A\to\DerCat[b]{\A}$ induces an equivalence of
  $\infty$-categories
  \[
    \Fun<ex>{\DerCat[b]{\A}}{\C}\stackrel{\sim}{\longrightarrow}\Fun<ex>{\A}{\C},
  \]
  where $\Fun<ex>{\A}{\C}$ denotes the $\infty$-category of functors $\A\to\C$
  that send (admissible) short exact sequences in $\A$ to fibre-cofibre
  sequences in $\C$.
\end{theorem}

\begin{remark}
  A generalisation of \Cref{thm:uniprop-Db} to the more general class of exact
  $\infty$-categories in the sense of Barwick~\cite{Bar15a} has been obtained
  in~\cite[Theorem~1.2]{Kle22}. The corresponding result for the closely-related
  class of exact differential graded categories is given
  in~\cite[Theorem~3.1]{Che24}.
\end{remark}

\begin{example}
  Let $\C$ be an essentially small stable $\infty$-category equipped with a
  $t$-structure $(\C_{\geq0},\C_{\leq0})$. Applying \Cref{thm:uniprop-Db} to the
  abelian category $\C^\heartsuit$ yields an equivalence of $\infty$-categories
  \[
    \Fun<ex>{\DerCat[b]{\C^\heartsuit}}{\C}\stackrel{\sim}{\longrightarrow}\Fun<ex>{\C^\heartsuit}{\C}.
  \]
  Choosing a pre-image of the inclusion functor $\C^\heartsuit\hookrightarrow\C$, we obtain an
  $\infty$-categorical realisation functor
  \[
    \operatorname{real}_t\colon\DerCat[b]{\C^\heartsuit}\longrightarrow\C.
  \]
  It is also worth mentioning that, in view of
  \Cref{prop:equivalences_of_stable_categories}, the usual criteria for
  verifying whether the above realisation functor is an equivalence (see for
  example~\cite[Theorem~3.11]{PV18}) remain valid in this context.
\end{example}

\subsection{Recollements of stable $\infty$-categories}

Similar to $t$-structures, the theory of recollments of triangulated
categories~\cite{CPS88} extends to stable $\infty$-categories almost verbatim.
We recall the necessary definitions.

\begin{defprop}[{\cite[Definition~A.8.1, Remark~A.8.19]{Lur17}}]
  \label{defprop:recollement}
  A \emph{recollement} of stable $\infty$-categories is a diagram
  \[
      \begin{tikzcd}
        \A\ar[hookrightarrow]{r}[description]{i}&\C\ar{r}[description]{p}\ar[shift
        right=0.5em]{l}[swap]{i_L}\ar[shift left=0.5em]{l}{i_R}&\B\ar[hook',shift left=0.5em]{l}{p_R}\ar[shift right=0.5em,hookrightarrow]{l}[swap]{p_L}
      \end{tikzcd}
    \]
    of stable $\infty$-categories and exact functors between them such that the
    following conditions hold:
    \begin{enumerate}
    \item There are adjoint triples $i_L\dashv i\dashv i_R$ and $p_L\dashv
      p\dashv p_R$. Moreover, the functors $i$, $p_L$ and $p_R$ are fully
      faithful.
    \item An object of $\C$ lies in the essential image of $i$ if and only if
      it lies in the kernel of $p$.
    \end{enumerate}
    Moreover, under the presence of the first condition, the second condition
    is equivalent to the following:
    \begin{enumerate}
    \item[(2')]\label{it:2prime} There is a natural isomorphism $p\circ
      i\simeq0$ and the functors $i_R$ and $p$ are jointly conservative, that is
      a morphism $f$ in $\C$ is invertible if and only if the morphisms $i_R(f)$
      and $p(f)$ are invertible.\footnote{Passing to opposite
        $\infty$-categories and functors, one may instead require the functors
        $i_L$ and $p$ to be jointly conservative.}
    \end{enumerate}
    In this context, we call the functor $F\coloneqq i_R\circ p_L\colon\B\to\A$
    the \emph{gluing functor} of the recollement.
  \end{defprop}
  \begin{proof}
    The equivalence between these two conditions is well known; we include a
    proof for the convenience of the reader since (the opposite of) condition
    (2') is used in \cite[Definition~A.8.1]{Lur17}. Suppose first that the
    condition that an object $\C$ lies in the essential image of $i$ if and only
    if it lies in the kernel of $p$ holds. We wish to show that the functors
    $i_R$ and $p$ are jointly conservative. For this, we observe that there is a
    fibre-cofibre sequence\footnote{The existence of this fibre-cofibre sequence
      can be established with an argument analogous to the one used to construct
      the corresponding decomposition triangles of a recollement of triangulated
      categories, keeping in mind that (co)limits in functor $\infty$-categories are
      computed pointwise.}
  \[
    \begin{tikzcd}
      ii_R\rar\dar\ar[phantom]{dr}[description]{\square}&\mathbf{1}_{\C}\dar\\
      0\rar&p_Rp
    \end{tikzcd}
  \]
  in the stable $\infty$-category
  $\Fun<ex>{\C}{\C}$ involving the counit and
  unit of the adjunctions $i\dashv i_R$ and $p\dashv p_R$, respectively. Since
  all functors involved are exact, given a morphism $f$ in $\C$, there is a
  fibre-cofibre sequence in $\C$ of the form
  \[
    \begin{tikzcd}
      ii_R(\cofib[f])\simeq\cofib[ii_R(f)]\rar\dar\ar[phantom]{dr}[description]{\square}&\cofib[f]\dar\\
      0\rar&\cofib[p_Rp(f)]\simeq p_Rp(\cofib[f])
    \end{tikzcd}
  \]
  The claim follows from the following elementary fact: A morphism in a stable
  $\infty$-category is invertible if and only if its cofibre
  vanishes.\footnote{Indeed, the equivalence between these conditions can be
    verified at the level of the triangulated homotopy category, where the
    equivalence is known (and elementary to show).} In view
  of the latter fibre-cofibre sequence, we conclude that $f$ is invertible if and
  only if the morphisms $ii_R(f)$ and $p_Rp(f)$ are invertible if and only if the morphisms
  $i_R(f)$ and $p(f)$ are invertible (since the functors $i$ and $p_R$ are fully
  faithful).

  Conversely, suppose that there is a natural isomorphism $p\circ i\simeq0$ and
  that the functors $i_R$ and $p$ are jointly conservative. Given an object $x\in\C$,
  consider the counit map $\varepsilon_x\colon ii_R(x)\to x$ in $\C$. Since $i$
  is fully faithful, the map $i_R(\varepsilon_x)$ is invertible in $\A$. From
  the condition
  $p\circ i\simeq 0$, it follows that $p(\varepsilon_x)$ is invertible if and
  only if $p(x)\simeq 0$. Thus, using the assumption that the functors $i_R$ and
  $p$ are jointly conservative, we conclude that $x$ lies in the essential image
  of $i$ if and only if the counit map $\varepsilon_x$ is invertible if and only
  if $x$ lies in the kernel of $p$.
\end{proof}

In the $\infty$-categorial context, the middle term of a recollement can be
reconstructed from its gluing functor. The key ingredient for this is the
following construction, which is analogue of \cite[4.1, Lemma~4.4,
Proposition~4.10]{KL15} that deals with semi-orthogonal decompositions of
pre-triangulated differential graded categories.

\begin{defprop}[{\cite[Proposition~A.8.7]{Lur17}}]
  \label{defprop:upper-triangular_gluing}
  Let $F\colon\B\to\A$ be an exact functor between stable $\infty$-categories.
  Define the upper-triangular gluing $\B\vec{\times}_F\A$ by means of the
  existence of cartesian square
  \[
    \begin{tikzcd}
      \B\vec{\times}_F\A\rar\dar\ar[phantom]{dr}[description]{\textup{PB}}&\Fun{s\to t}{\A}\dar{s^*}\\
      \B\rar[swap]{F}&\A
    \end{tikzcd}
  \]
  in the $\infty$-category $\cat<ex>$, so that the objects of
  $\B\vec{\times}_F\A$ can be identified with pairs $(b,\ \varphi\colon
  F(b)\to a)$ consisting of an object $b\in\B$ and a morphism $\varphi$ in $\A$.
  Then, the stable $\infty$-category
  $\B\vec{\times}_F\A$ is part of a recollement
  \[
    \begin{tikzcd}
\A\ar[hookrightarrow]{r}[description]{i}&\B\vec{\times}_F\A\ar{r}[description]{p}\ar[shift
        right=0.5em]{l}[swap]{i_L}\ar[shift left=0.5em]{l}{i_R}&\B\ar[hook',shift left=0.5em]{l}{p_R}\ar[shift right=0.5em,hookrightarrow]{l}[swap]{p_L}
    \end{tikzcd}
  \]
  in which, at the level of objects, the various functors are given as follows:
  \begin{align*}
    i_L(b,\varphi\colon F(b)\to a)&=\cofib[\varphi]\\
    i(a)&=(0,F(0)\simeq0\to a)\\
    i_R(b,\varphi\colon F(b)\to a)&=a\\
    p_L(b)&=(b,\id\colon F(b)\to F(b))\\
    p(b,\varphi\mathcal\colon F(b)\to a)&=b\\
    p_R(b)&=(b,F(b)\to0)
  \end{align*}
  
\end{defprop}

\begin{remark}
  In \Cref{defprop:upper-triangular_gluing}, one may instead consider the
  \emph{lower-triangular gluing} $\A\vec{\times}_F\B$, which is defined by means of the
  existence of cartesian square
  \[
    \begin{tikzcd}
      \A\vec{\times}_F\B\rar\dar\ar[phantom]{dr}[description]{\text{PB}}&\Fun{s\to t}{\A}\dar{t^*}\\
      \B\rar[swap]{F}&\A
    \end{tikzcd}
  \]
  in the $\infty$-category $\cat<ex>$, and whose objects are pairs
  $(b,\psi\colon a\to F(b))$ consisting of an object $b\in\B$ and a morphism
  $\psi$ in $\A$. Remarkably, there is an equivalence of
  stable $\infty$-categories
  \[
    \A\vec{\times}_F\B\stackrel{\sim}{\longrightarrow}\B\vec{\times}_F\A\] that,
  ultimately, is induced by the autoequivalence
  \[
    \operatorname{Cofib}\colon\Fun{s\to t}{\A}\stackrel{\sim}{\longrightarrow}\Fun{s\to
      t}{\A},\quad (f\colon a'\to a)\longmapsto(a\to\cofib[f]),
  \]
  see \cite[Lemma~1.3]{DJW21} for the (easy) proof. In~\cite{DJW21}, the
  above equivalence is used to give a very general construction of reflection
  functors based on seminal ideas of Ladkani~\cite{Lad08} and
  Groth--{\v{S}}\v{t}ov\'{i}\v{c}ek~\cite{GS16,GS16a,GS18a,GS18b}.
\end{remark}

\begin{example}
  Let $\C$ be a stable $\infty$-category. Applying
  \Cref{defprop:upper-triangular_gluing} to the identity functor of $\C$ yields
  a recollement
  \[
      \begin{tikzcd}
        \C\ar[hookrightarrow]{r}[description]{i}&\Fun{s\to t}{\C}\ar{r}[description]{p}\ar[shift
        right=0.5em]{l}[swap]{i_L}\ar[shift left=0.5em]{l}{i_R}&\C\ar[hook',shift left=0.5em]{l}{p_R}\ar[shift right=0.5em,hookrightarrow]{l}[swap]{p_L}
      \end{tikzcd}
  \]
  whose middle term is the (stable) $\infty$-category of morphisms in $\C$ and
  in which the functor $i_L$ is given by the cofibre functor
  \[
    \cofib\colon\Fun{s\to t}{\C}\longrightarrow\C,\qquad\varphi\longmapsto\cofib[\varphi].
  \]
  In this case, the functor $i_R$ is left adjoint to $p_L$ and the
  above recollement extends to an infinite ladder of recollements, compare with
  \cite[Proposition~C.3]{Kra20a} that deals with morphic enhancements of
  triangulated categories.
\end{example}

\begin{theorem}[{\cite[A.8.11]{Lur17}}]
  \label{thm:regluing}
    Consider a recollement of stable $\infty$-categories
  \[
    \begin{tikzcd}
      \A\ar[hookrightarrow]{r}[description]{i}&\C\ar{r}[description]{p}\ar[shift
      right=0.5em]{l}[swap]{i_L}\ar[shift left=0.5em]{l}{i_R}&\B\ar[hook',shift left=0.5em]{l}{p_R}\ar[shift right=0.5em,hookrightarrow]{l}[swap]{p_L}
    \end{tikzcd}
  \]
  with gluing functor ${F\coloneq i_R\circ p_L\colon\B\to\A}$.
  Then, there exists a commutative diagram
    \[
    \begin{tikzcd}
      \A\ar[hookrightarrow]{r}[description]{i}\dar{\id}&\C\dar[description]{\wr}\ar{r}[description]{p}\ar[shift
      right=0.5em]{l}[swap]{i_L}\ar[shift left=0.5em]{l}{i_R}&\B\ar[hook',shift
      left=0.5em]{l}{p_R}\ar[shift right=0.5em,hookrightarrow]{l}[swap]{p_L}\dar{\id}\\
       \A\ar[hookrightarrow]{r}[description]{i}&\B\vec{\times}_{F}\A\ar{r}[description]{p}\ar[shift
      right=0.5em]{l}[swap]{i_L}\ar[shift left=0.5em]{l}{i_R}&\B\ar[hook',shift left=0.5em]{l}{p_R}\ar[shift right=0.5em,hookrightarrow]{l}[swap]{p_L}
    \end{tikzcd}
  \]
  in which the functor
  \[
    \C\stackrel{\sim}{\longrightarrow}\B\vec{\times}_F\A,\qquad
    x\longmapsto(p(x),\ i_R(\varepsilon_x)\colon i_R(p_Lp(x))=F(p(x))\to i_R(x))
  \]
  is an equivalence of stable $\infty$-categories, where $\varepsilon\colon p_Lp\to\mathbf{1}_{\C}$ is the counit of the adjunction
  $p_L\dashv p$.
\end{theorem}

\begin{example}
  Let $R$ and $S$ be rings and $M$ and $R$-$S$-bimodule. The derived
  $\infty$-category of the upper-triangular matrix ring
  $\left(\begin{smallmatrix}R&M\\0&S\end{smallmatrix}\right)$ participates in a
  recollment of stable $\infty$-categories
  \[
    \begin{tikzcd}
      \DerCat{\Mod{S}}\ar[hookrightarrow]{r}[description]{i}&\operatorname{D}\left(\begin{smallmatrix}R&M\\0&S\end{smallmatrix}\right)\ar{r}[description]{p}\ar[shift
      right=0.5em]{l}[swap]{i_L}\ar[shift left=0.5em]{l}{i_R}&\DerCat{\Mod{R}}\ar[hook',shift left=0.5em]{l}{p_R}\ar[shift right=0.5em,hookrightarrow]{l}[swap]{p_L}
    \end{tikzcd}
  \]
  whose gluing functor identifies with
  $-\otimes_R^\mathbb{L}M\colon\DerCat{\Mod{R}}\to\DerCat{\Mod{S}}$. It follows
  from \Cref{thm:regluing} that
  there is an equivalence of stable $\infty$-categories
  \[
    \operatorname{D}\left(\begin{smallmatrix}R&M\\0&S\end{smallmatrix}\right)\simeq\DerCat{\Mod{R}}\vec{\times}\DerCat{\Mod{S}}.
  \]
  This equivalence is leveraged in \cite{Jas24a} for extending the main results
  in~\cite{Lad11,May11}. For a a generalisation to `square-zero'
  upper-triangular gluings with more terms, we refer the reader
  to~\cite[Proposition~2.39]{Chr22}.
\end{example}

The upper-triangular gluing operation introduced in
\Cref{defprop:upper-triangular_gluing} is an instance of a lax
limit~\cite{GHN17}, see \cite[Lemma~5.4.7.15]{Lur09}. The following
self-referential universal property can be seen as a manifestation of this fact.

\begin{proposition}
  \label{prop:lax-limit}
  Let $F\colon\B\to\A$ be an exact functor between stable $\infty$-categories.  For each stable $\infty$-category $\C$, there is a canonical equivalence of
  stable $\infty$-categories
  \[
    \Fun<ex>{\C}{\B\vec{\times}_F\A}\stackrel{\sim}{\longrightarrow}\Fun<ex>{\C}{\B}\vec{\times}_{F_*}\Fun<ex>{\C}{\A},\quad\Phi\longmapsto(p_*(\Phi),(i_R)_*(\varepsilon_\Phi)),
  \]
  where $F_*$ denotes the post-composition functor
  \[
    F_*\colon\Fun<ex>{\C}{\B}\longrightarrow\Fun<ex>{\C}{\A},\qquad
    \Phi\longmapsto F\circ\Phi,
  \]
  and where $\varepsilon\colon p_Lp\to\mathbf{1}_{\C}$ is the counit of the adjunction
  $p_L\dashv p$.
\end{proposition}
\begin{proof}
  The existence of the claimed equivalence follows immediately from the
  existence of the induced cartesian square
  \[
    \begin{tikzcd}
      \Fun<ex>{\C}{\B\vec{\times}_F\A}\rar\dar[swap]{p_*}\ar[phantom]{dr}[description]{\text{PB}}&\Fun{s\to t}{\Fun<ex>{\C}{\A}}\dar{s^*}\\
      \Fun<ex>{\C}{\B}\rar[swap]{F_*}&\Fun<ex>{\C}{\A},
    \end{tikzcd}
  \]
  where we implicitly use the canonical identification
  \[
    \Fun<ex>{\C}{\Fun{s\to t}{\A}}\simeq\Fun{s\to t}{\Fun<ex>{\C}{\A}}.\qedhere
  \]
\end{proof}

Perhaps surprisingly, the upper-triangular gluing construction is also an
instance of a lax colimit, as witnessed by the following dual universal
property.

\begin{proposition}
  \label{prop:lax-colimit}
  Let $F\colon\B\to\A$ be an exact functor between stable $\infty$-categories.
  For each stable $\infty$-category $\C$, there is a canonical equivalence of
  stable $\infty$-categories
  \[
    \Fun<ex>{\B\vec{\times}_F\A}{\C}\stackrel{\sim}{\longrightarrow}\Fun<ex>{\A}{\C}\vec{\times}_{F^*}\Fun<ex>{\B}{\C},\quad
    \Phi\longmapsto(i^*(\Phi),\ \Phi(\varepsilon_{p_L})),
  \]
  where $F^*$ denotes the pre-composition functor
  \[
    F^*\colon\Fun<ex>{\A}{\C}\longrightarrow\Fun<ex>{\B}{\C},\qquad
    \Phi\longmapsto \Phi\circ F,
  \]
  and $\varepsilon\colon ii_R\to\mathbf{1}_{\C}$ is the counit of the adjunction
  $i\dashv i_R$.
\end{proposition}
\begin{proof}
  The statement is well-known to experts; it is also a special instance of the a
  much more general `finite lax additivity' property~\cite[Example~5.7]{CDW24}.
  We provide a short proof for the convenience of the reader. The canonical
  recollement
  \[
    \begin{tikzcd}
\A\ar[hookrightarrow]{r}[description]{i}&\B\vec{\times}_F\A\ar{r}[description]{p}\ar[shift
        right=0.5em]{l}[swap]{i_L}\ar[shift left=0.5em]{l}{i_R}&\B\ar[hook',shift left=0.5em]{l}{p_R}\ar[shift right=0.5em,hookrightarrow]{l}[swap]{p_L}
    \end{tikzcd}
  \]
  induces a diagram of functor
  categories\footnote{Using~\cite[Proposition~5.2.2.8]{Lur09} and its dual, the usual proof applies.}
  \[
    \begin{tikzcd}
\Fun{\B}{\C}\ar[hookrightarrow]{r}[description]{p^*}&\Fun{\B\vec{\times}_F\A}{\C}\ar{r}[description]{i^*}\ar[shift
        right=0.5em]{l}[swap]{p_R^*}\ar[shift left=0.5em]{l}{p_L^*}&\Fun{\A}{\C}\ar[hook',shift left=0.5em]{l}{i_L^*}\ar[shift right=0.5em,hookrightarrow]{l}[swap]{i_R^*}
    \end{tikzcd}
  \]
  that we claim restricts to a recollement
  \[
    \begin{tikzcd}
\Fun<ex>{\B}{\C}\ar[hookrightarrow]{r}[description]{p^*}&\Fun<ex>{\B\vec{\times}_F\A}{\C}\ar{r}[description]{i^*}\ar[shift
        right=0.5em]{l}[swap]{p_R^*}\ar[shift left=0.5em]{l}{p_L^*}&\Fun<ex>{\A}{\C}\ar[hook',shift left=0.5em]{l}{i_L^*}\ar[shift right=0.5em,hookrightarrow]{l}[swap]{i_R^*}.
    \end{tikzcd}
  \]
  whose gluing functor is $F^*=(i_R\circ p_L)^*=p_L^*\circ i_R^*$.

  Since the composition of exact functors is exact, it only remains to show that
  the kernel of $i^*$ agrees with the essential image of $p^*$. For this,
  observe that the adjunction $p\dashv p_R$ exhibits $\B$ as the localisation of
  $\B\vec{\times}_F\A$ at the class of morphisms that are sent to isomorphisms
  by $p$~\cite[Proposition~5.2.7.12]{Lur09}, which, since the functor $p$ is
  exact, is precisely the class of morphisms whose cone lies in the essential
  image of $i$. In other words, according to \Cref{defprop:Verdier_quotient},
  the functor $p$ exhibits $\B$ as the Verdier quotient of $\B\vec{\times}_F\A$
  by the essential image of $i$ and, therefore, the functor
  \[
    p^*\colon\Fun<ex>{\B}{\C}\longrightarrow\Fun<ex>{\B\vec{\times}_F\A}{\C}
  \]
  is fully faithful and its essential image is the kernel of the restriction
  functor $i^*$. This finishes the proof.
\end{proof}

The universal propertites of the upper-triangular gluing established in
\Cref{prop:lax-limit,prop:lax-colimit} have interesting consequences, among
which we highlight the following observation. Below, the mapping spectrum of
endomorphisms of the idendity functor should be thought as an avatar of
topological Hochschild cohomology.\footnote{Compare with~\cite[Remarks~D.1.5.6
and~D.1.5.7]{Lur18SAG}.}

\begin{proposition}
  \label{prop:HH-sq}
  Consider a recollement of stable $\infty$-categories
  \[
      \begin{tikzcd}
        \A\ar[hookrightarrow]{r}[description]{i}&\C\ar{r}[description]{p}\ar[shift
        right=0.5em]{l}[swap]{i_L}\ar[shift left=0.5em]{l}{i_R}&\B\ar[hook',shift left=0.5em]{l}{p_R}\ar[shift right=0.5em,hookrightarrow]{l}[swap]{p_L}
      \end{tikzcd}
    \]
    with gluing functor $F\coloneqq i_R\circ p_L\colon\B\to\A$. Then, there is a
    bicartesian square\footnote{The reader should compare this bicartesian
    square with Happel's long exact sequence for the Hochschild cohomology of a
    one-point (co)extension~\cite[Theorem~5.3]{Hap89} and Keller's vast generalisation
    thereof~\cite[Lemma~4.5]{Kel03} (see also~\cite[Theorem~7.7]{Kuz09}).}
  \begin{equation*}
    \label{eq:HH-square}
    \begin{tikzcd}
      \MapSp{\mathbf{1}_{\C}}{\mathbf{1}_{\C}}\rar\dar\ar[phantom]{dr}[description]{\square}&\MapSp{\mathbf{1}_\A}{\mathbf{1}_\A}\dar\\
      \MapSp{\mathbf{1}_{\B}}{\mathbf{1}_{\B}}\rar&\MapSp{F}{F}
    \end{tikzcd}
  \end{equation*}
  in the $\infty$-category of spectra, where each of the mapping spectra is computed in the relevant
  $\infty$-category of functors.
\end{proposition}
\begin{proof}
\Cref{thm:regluing} and \Cref{prop:lax-colimit} yield an equivalence of stable $\infty$-categories
\[
  \Fun<ex>{\C}{\C}\stackrel{\sim}{\longrightarrow}\Fun<ex>{\A}{\C}\vec{\times}_{F^*}\Fun<ex>{\B}{\C},\qquad\Phi\longmapsto
  (i^*(\Phi),\Phi(\varepsilon_{p_L})),
\]
that sends $\mathbf{1}_{\C}$ to the pair $(i,\varepsilon_{p_L})$ and, therefore,
there is an induced equivalence of mapping spectra\footnote{This is special case
  of the general formula for computing mapping spaces/mapping spectra in an
  upper-triangular gluing, compare with \cite[Remark~4.1]{KL15}.}
\[
  \MapSp{\mathbf{1}_\C}{\mathbf{1}_\C}\stackrel{\sim}{\to}\MapSp{(i,\varepsilon_{p_L})}{(i,\varepsilon_{p_L})}=\MapSp{i}{i}\times_{\MapSp{F^*(i)}{p_L}}\MapSp{p_L}{p_L}.
\]
In this way, we obtain a bicartesian square of mapping
spectra of the form
\[
  \begin{tikzcd}
    \MapSp{\mathbf{1}_{\C}}{\mathbf{1}_{\C}}\rar\dar\ar[phantom]{dr}[description]{\square}&\MapSp{i}{i}\dar\\
    \MapSp{p_L}{p_L}\rar&\MapSp{F^*(i)}{p_L},
  \end{tikzcd}
\]
which we claim is already the desired bicartesian square.
To see this, observe that the fully faithful functor
\[
  i_*\colon\Fun<ex>{\A}{\A}\longrightarrow\Fun<ex>{\A}{\C}
\]
induces an equivalence of mapping spectra
\[
  \MapSp{\mathbf{1}_\A}{\mathbf{1}_\A}\stackrel{\sim}{\longrightarrow}\MapSp{i}{i}.
\]
Similarly, there is an equivalence of mapping spectra
\[
  \MapSp{\mathbf{1}_\B}{\mathbf{1}_\B}\stackrel{\sim}{\longrightarrow}\MapSp{p_L}{p_L}
\]
induced by the fully faithful functor
\[
  (p_L)_*\colon\Fun<ex>{\B}{\B}\longrightarrow\Fun<ex>{\B}{\C}.
\]
Finally, the adjunction $i\dashv i_R$ affords a further equivalence of mapping
spectra
\[
  \MapSp{F^*(i)}{p_L}=\MapSp{iF}{p_L}\simeq\MapSp{F}{i_Rp_L}=\MapSp{F}{F}.
\]
This finishes the proof.
\end{proof}


\begin{acknowledgements}
  The author thanks the (other) members of the Organising Committe and of the
Scientific Committee of the tenth edition of the conference Advances in
Representation Theory of Algebras for the invitation to contribute this survey
to the proceedings of the conference. The author also thanks Xiaofa Chen, Merlin
Christ, Omar G{\'o}mez, Marvin Plogmann, Greg Stevenson and an anonymous referee for their comments on a previous version of the
survey.


\end{acknowledgements}

\printbibliography

\end{document}
